\documentclass[oneside,11pt]{amsart}
\usepackage{amssymb,mathrsfs,amstext, comment, xcolor}
\usepackage{enumitem}
\usepackage{amsmath, mathtools}
\usepackage[noadjust]{cite}
\usepackage[%backref,
plainpages=false,colorlinks,hyperindex,pdfpagemode=None,bookmarksopen,linkcolor=blue,citecolor=blue,urlcolor=blue]{hyperref}

\newtheorem{theorem}{Theorem}[section]
\newtheorem{proposition}[theorem]{Proposition}
\newtheorem{corollary}[theorem]{Corollary}
\newtheorem{lemma}[theorem]{Lemma}

\theoremstyle{definition}

\newtheorem{remark}[theorem]{Remark}

\def\cB{{\mathcal B}}

\def\cD{{\mathcal D}}

\def\cS{{\mathcal S}}

\def\cC{{\mathcal C}}

\def\diag{{\rm diag}\,}
\def\tr{{\rm tr}\,}

\def\Eig(\rm Eig{\,}

\def\Eig{{\rm Eig}\,}
\def\qed{\hfill\vbox{\hrule width 6 pt\hbox{\vrule height 6 pt width 6
pt}}\medskip}

\def\cB{{\mathcal B}}

\def\IF{{\mathbb F}}

\def\IF{{\mathbb F}}

\def\diag{{\rm diag}}

\def\char{{\rm char}}

\begin{document}

\openup .84\jot

\title{Linear maps preserving product of involutions}
\author[Li]{Chi-Kwong Li}
\author[Lohan]{Tejbir Lohan}
\author[Singla]{Sushil Singla}

\makeatletter
\@namedef{subjclassname@2020}{\textup{2020} Mathematics Subject Classification}
\makeatother



\begin{abstract}
	An element of the algebra $M_n(\mathbb{F})$ of $n\times n$ matrices over a field $\mathbb{F}$ is called an involution if its square equals the identity matrix. Gustafson, Halmos, and Radjavi proved that any product of involutions in $M_n(\mathbb{F})$ can be expressed as a product of at most four involutions. In this article, we investigate the bijective linear preservers of the sets of products of two, three, or four involutions in $M_n(\mathbb{F})$. 
\end{abstract}

\maketitle

\section{Introduction}

Let $M_n(\IF)$ be the algebra of $n \times n$ matrices over a field $\IF$. Two matrices $A,B \in M_n(\IF)$ are said to be similar if $A = PBP^{-1}$ for some invertible matrix $P$ in $M_n(\IF)$. A matrix $A\in M_n(\IF)$ is called an involution if $A^2=I$, where $I$ denotes the identity matrix in $M_n(\IF)$. Decomposing a group element into a product of involutions has applications in various areas of mathematics, with particular emphasis on elements that can be written as products of two involutions, known as \textit{bireflectional}, \textit{strongly reversible}, or \textit{strongly real} elements; see \cite{Wo, Gu, OS}. We note that any product of involutions in $M_n(\IF)$ has determinant $\pm 1$. In \cite{GHR}, Gustafson, Halmos, and Radjavi proved that a matrix in $M_n(\IF)$ has determinant $\pm 1$ if and only if it is a product of at most four involutions. It is also known that a matrix in $M_n(\IF)$ is a product of two involutions if and only if it is similar to its inverse. Wonenburger \cite{Wo} first proved this result in the case where the characteristic of $\IF$ is not two, and the general case was later independently proved by Djokovi\'{c} \cite{Dj} and by Hoffman and Paige \cite{HP}. The characterization of products of three involutions is less well understood in the literature; see \cite{GHR, Ba, Liu} for a survey of known results.

Linear preserver problems constitute an active area of research in matrix theory and operator theory, focusing on the characterization of linear maps between algebras that preserve specific subsets, functions, or relations. These problems are of broad interest and have been extensively studied; see \cite{GLS, LP, LT} and the references therein for a detailed exposition. In \cite{LTWW}, the authors studied linear maps that preserve matrices annihilated by a fixed polynomial with distinct roots. From their results, it follows that a bijective linear preserver of the set of involutions in $M_n(\IF)$ has the form $X \mapsto \pm PXP^{-1}$ or $X \mapsto \pm PX^tP^{-1}$, where $X^t$ denotes the transpose of $X$, provided that the characteristic of the field $\IF$ is not two. For the case of characteristic two, we refer to \cite[Chapter~9]{ZTC}. In this paper, we investigate bijective linear preservers of the sets of products of involutions in $M_n(\IF)$ for $n \ge 2$.

A linear map $T: M_n(\IF) \rightarrow M_n(\IF)$ is called \textit{unital} if $T(I)=I$. Let $\cS_n := \cB_n$, $\cC_n$, or $\cD_n$, where $\cB_n$, $\cC_n$, and $\cD_n$ denote the sets of all matrices in $M_n(\IF)$ that are products of two, three, and four involutions, respectively. In the following result, we characterize the bijective linear maps from $M_n(\IF)$ to $M_n(\IF)$ that preserve the set $\cS_n$.

\begin{theorem}\label{thm-main}
	Let $\IF$ be a field with characteristic not equal to $2$. Let $T: M_n(\IF) \rightarrow M_n(\IF)$ be a bijective linear map. Then the following statements hold.
	\begin{enumerate}[label={\normalfont(\roman*)}]
		\item $T(\cB_n) \subseteq \cB_n$ if and only if there exist a scalar $\alpha \in \{-1,1\}$ and an invertible matrix $P\in M_n(\IF)$ such that $T$ has the form
		\begin{equation*}
			X \mapsto \alpha PXP^{-1} \quad \text{or} \quad X \mapsto \alpha PX^tP^{-1}.
		\end{equation*}
		
		\item $T$ is unital and $T(\cC_n) \subseteq \cC_n$ if and only if there exists an invertible matrix $P\in M_n(\IF)$ such that $T$ has the form
		\begin{equation*}
			X \mapsto PXP^{-1} \quad \text{or} \quad X \mapsto PX^tP^{-1}.
		\end{equation*}
		
		\item $T(\cD_n) \subseteq \cD_n$ if and only if there exist invertible matrices $P, Q\in M_n(\IF)$ with $\det(PQ)\in \{-1,1\}$ such that $T$ has the form
		\begin{equation*}
			X \mapsto PXQ \quad \text{or} \quad X \mapsto PX^tQ.
		\end{equation*}
	\end{enumerate}
\end{theorem}

As a direct consequence of the above theorem, all bijective unital linear preservers have the same form for $\cB_n$, $\cC_n$, and $\cD_n$. More precisely, we obtain the following result.

\begin{corollary}\label{cor-main}
	Let $\IF$ be a field with characteristic not equal to $2$. Let $T: M_n(\IF) \rightarrow M_n(\IF)$ be a unital bijective linear map. Then $T(\cS_n) \subseteq \cS_n$ if and only if there exists an invertible matrix $P\in M_n(\IF)$ such that $T$ has the form
	$$X \mapsto PXP^{-1} \quad \text{or} \quad X \mapsto PX^tP^{-1}.$$
\end{corollary}

We will prove Theorem~\ref{thm-main} in Section~\ref{proofs}. We note that, unlike $\cB_n$ and $\cD_n$, Theorem~\ref{thm-main} addresses only the bijective unital preservers of $\cC_n$. However, when $\cC_n=\cD_n$, Theorem~\ref{thm-main}(iii) can be applied to characterize the bijective linear preservers of $\cC_n$. The cases in which $\cC_n=\cD_n$ are characterized in \cite[Fact~5]{Ba} (see also Remark~\ref{rem_C_n}). In Section~\ref{sec-related-results-part-1}, we study bijective linear maps that preserve $\cC_n$ and show in Theorem~\ref{thm-non-unital-C_n} that such maps have the form 
$X \mapsto PXQ$ or $X \mapsto PX^tQ$ for some invertible matrices $P, Q \in M_n(\IF)$ such that $PQA \in \cC_n$ whenever $A \in \cC_n$. We then use Theorem~\ref{thm-non-unital-C_n} to describe the bijective linear preservers of $\cC_3$ in the case $\cC_3 \subsetneq \cD_3$; see Theorem~\ref{thm-prod-3-not2}.

Note that in Theorem~\ref{thm-main} we assume that the characteristic of the field is not $2$. When the field has characteristic $2$, however, the characterization of linear preservers of products of involutions becomes more delicate. In Section~\ref{sec-related-results-part-2}, we classify the bijective linear preservers of $\cS_2$ and $\cS_3$ over fields of characteristic $2$; see Theorems~\ref{thm-main-2} and~\ref{thm-main-order-3-char-2}. Finally, we conclude the paper by discussing potential directions for future research in Section~\ref{sec-remarks}.

%The paper is organized as follows. In Section~\ref{section2}, we present preliminary results on products of involutions. In Section~\ref{proofs}, we prove the main theorem. Section~\ref{sec-related-results} investigates non-unital bijective linear preservers of $\cC_n$ and discusses the cases $n=2,3$ when the field has characteristic $2$. In Section~\ref{sec-remarks}, we outline several directions for further research.

\section{Preliminaries}\label{section2}

This section introduces the required notation and discusses several preliminary results. Let $\IF$ be a field, and let $\mathrm{char}(\IF)$ denote its characteristic. Recall that $M_n(\IF)$ denotes the algebra of $n \times n$ matrices over $\IF$. Let $E_{i,j}\in M_n(\IF)$ denote the matrix with $1$ in position $(i,j)$ and zeros elsewhere. For $A\in M_n(\IF)$, let $\det(A)$ denote the determinant of $A$. The direct sum $A \oplus B$ of two matrices $A \in M_m(\IF)$ and $B \in M_n(\IF)$ is the block diagonal matrix in $M_{m+n}(\IF)$ defined by
$
A \oplus B :=
\begin{psmallmatrix}
	A & 0 \\
	0 & B
\end{psmallmatrix}.
$

Let $f(x)= x^n + a_{n-1}x^{n-1} + \dots + a_1 x + a_0$ be a monic polynomial of degree $n$ over $\IF$, where $a_i \in \IF$ for all $0\le i \le n-1$. The companion matrix $C_f \in M_n(\IF)$ associated with the monic polynomial $f$ is defined by
\begin{equation*}
	C_{f} :=
	\begin{psmallmatrix}
		0      & 0          & \cdots & 0 & -a_0  \\
		1      & 0           & \cdots & 0 & -a_1 \\
		0      & 1           & \cdots & 0 & -a_2 \\
		\vdots & \vdots  & \smash{\ddots} & \vdots & \vdots \\
		0   & 0   &  \cdots  & 1 & -a_{n-1}
	\end{psmallmatrix}
	\quad \text{for } n \ge 2,
	\quad \text{and} \quad
	C_{f}:= -a_0 \quad \text{for } n=1.
\end{equation*}

The following result recalls the well-known rational canonical form for matrices in \(M_n(\mathbb{F})\).
%; see \cite[Theorem 9.32]{Zhan} for a reference.

\begin{proposition}
	For every \(A \in M_n(\mathbb{F})\), there exists an invertible matrix 
	\(P \in M_n(\mathbb{F})\) such that
	\begin{equation}\label{eq-rational-form}
		PAP^{-1} = C_{f_1} \oplus C_{f_2} \oplus \cdots \oplus C_{f_r},
	\end{equation}
	where each \(f_i(x)\) is a positive power of a monic irreducible polynomial over \(\mathbb{F}\), and \(C_{f_i}\) denotes the companion matrix associated with \(f_i(x)\). Moreover, the decomposition \eqref{eq-rational-form} is uniquely determined by \(A\) up to a permutation of the companion matrices.
\end{proposition}

A matrix \(A \in M_n(\mathbb{F})\) is called \emph{indecomposable} if it is not similar to a nontrivial direct sum of two matrices. Equivalently, we note that \(A\) is indecomposable if and only if it is similar to the companion matrix \(C_{f^k}\) of \(f(x)^k\), where \(f(x) \in \mathbb{F}[x]\) is an irreducible polynomial and \(k \ge 1\).

Note that both the characteristic polynomial and the minimal polynomial of $C_f$ coincide and are given by $\det(\lambda I - C_f)=f(\lambda)$. Since $\det(C_f)=(-1)^n a_0$, the matrix $C_f$ is invertible if and only if $a_0\neq0$. Moreover, if $a_0\neq0$, then
\begin{equation*}
	C_{f}^{-1} =
	\begin{psmallmatrix}
		- a_0^{-1} a_1   & 1      & 0      & \cdots & 0 \\
		- a_0^{-1} a_2   & 0      & 1      & \cdots & 0 \\
		\vdots & \vdots & \vdots & \smash{\ddots} & \vdots \\
		- a_0^{-1} a_{n-1}  & 0      & 0      & \cdots & 1 \\
		- a_0^{-1}  & 0   & 0 &  \cdots  & 0
	\end{psmallmatrix}.
\end{equation*}

This motivates the definition of the \textit{reciprocal} $\widetilde{f}(x)$ of a monic polynomial $f(x)$ with nonzero constant term. Let $f(x)=x^n+a_{n-1}x^{n-1}+\dots+a_1x+a_0$ be a monic polynomial over $\IF$ with $a_0\neq0$. The reciprocal polynomial $\widetilde{f}(x)$ is defined by
\[
\widetilde{f}(x):=a_0^{-1}x^n f(x^{-1})
= x^n + a_0^{-1}a_1x^{n-1} + \dots + a_0^{-1}a_{n-1}x + a_0^{-1}.
\]

A monic polynomial $f(x)\in\IF[x]$ with nonzero constant term is called \textit{self-reciprocal} (or \textit{symmetric}) if $f(x)=\widetilde{f}(x)$; otherwise, it is called \textit{non-self-reciprocal}. Moreover, for any $k\ge1$, the polynomial $f(x)^k$ is self-reciprocal if and only if $f(x)$ is self-reciprocal.

Note that $\widetilde{f}(x)$ is monic, and $\widetilde{f}(\alpha^{-1})=0$ if and only if $f(\alpha)=0$ for any nonzero $\alpha\in\IF$. Furthermore, $f(x)=\widetilde{f}(x)$ if and only if $a_0=a_0^{-1}$ and $a_{n-k}=a_0^{-1}a_k$ for all $1\le k\le n-1$.

Let $\tau_n=[\tau_{i,j}]$ be the involution in $M_n(\IF)$ defined by $\tau_{i,j}=1$ if $i+j=n+1$ and $\tau_{i,j}=0$ otherwise. Then, for a monic polynomial $f(x)=x^n+a_{n-1}x^{n-1}+\dots+a_1x+a_0$ over $\IF$ with $a_0\neq0$, we have $\tau_n C_f^{-1}\tau_n^{-1}=C_{\widetilde{f}}$, where $C_{\widetilde{f}}$ is the companion matrix corresponding to the reciprocal polynomial $\widetilde{f}(x)$. Hence,
\begin{equation*}
	\begin{psmallmatrix}
		& \tau_n  \\
		\tau_n &
	\end{psmallmatrix}
	\begin{psmallmatrix}
		C_f & 0 \\
		0 & C_{\widetilde{f}}
	\end{psmallmatrix}^{-1}
	\begin{psmallmatrix}
		& \tau_n  \\
		\tau_n &
	\end{psmallmatrix}^{-1}
	=
	\begin{psmallmatrix}
		C_f & 0 \\
		0 & C_{\widetilde{f}}
	\end{psmallmatrix}.
\end{equation*}

This shows that the matrix $\begin{psmallmatrix} C_f & 0 \\ 0 & C_{\widetilde{f}} \end{psmallmatrix}$ belongs to $\cB_{2n}$, where $n$ is the degree of the polynomial $f(x)$. The following result, which follows from \cite[Lemma~1]{HP}, classifies the elements of $\cB_n$; see also \cite[p.~244]{Gu}.

\begin{proposition}\label{prop-rev-class}
	Let $A\in M_n(\IF)$ be an invertible matrix. Then $A \in \cB_n$ if and only if, up to similarity, the rational canonical form of $A$ has the form
	\begin{equation}\label{eq-rev-class}
		A = B \oplus C,
	\end{equation}
	where $B=\bigoplus_{i=1}^{s} C_{f_i}$ and $C=\bigoplus_{j=1}^{t}\big(C_{g_j}\oplus C_{\widetilde{g_j}}\big)$ such that each $f_i$ is a power of a self-reciprocal irreducible monic polynomial and each $g_j$ is a power of a non-self-reciprocal irreducible monic polynomial.
\end{proposition}

A matrix $A \in M_n(\IF)$ is called \textit{nilpotent} if there exists an integer $k\ge1$ such that $A^k=O$, where $O$ denotes the $n\times n$ zero matrix. Let $\mathcal N$ denote the set of all nilpotent matrices in $M_n(\IF)$. By Proposition~\ref{prop-rev-class}, we have $\pm I+tN\in\cB_n$ for all $t\in\IF$ and $N\in\mathcal N$.

We now establish structural results concerning linear preservers of $\cB_n$. In the following results, for certain matrices $A\in\cB_n$, we construct matrices $X\in\cB_n$ such that $AX\notin\cB_n$. These constructions are used in Lemma~\ref{lem-key-nonunital-B} to show that, under additional assumptions, any bijective linear preserver of $\cB_n$ must satisfy $T(I)=\pm I$. This constitutes a key step in the proof of Theorem~\ref{thm-main}(i).

\begin{lemma}\label{lem-special-nil-comp-type2}
	Let $n \ge 2$ and let $g(x)$ be a monic polynomial of degree $n$ with nonzero constant term that is not a power of a self-reciprocal irreducible polynomial. Suppose that $A= C_g \oplus C_{\widetilde g} \in \cB_{2n}$, where $C_g$ and $C_{\widetilde g}$ are the companion matrices of $g(x)$ and $\widetilde g(x)$, respectively. Then there exists a strictly upper triangular matrix $N$ such that $A(I+N)\notin \cB_{2n}$.
\end{lemma}

\begin{proof}
	Let 
	$g(x)=x^n+a_{n-1}x^{n-1}+\cdots+a_1x+a_0$, and let 
	$\widetilde g(x)=x^n+a_0^{-1}a_1x^{n-1}+\cdots+a_0^{-1}a_{n-1}x+a_0^{-1}$ denote its reciprocal polynomial.
	
	Consider the strictly upper triangular matrix $N=O\oplus E_{1,n}\in M_{2n}(\IF)$, whose only nonzero entry is $1$ in the $(n+1,2n)$ position. Then
	\[
	A(I + N)= (C_{g} \oplus C_{\widetilde{g}}) \big((I + O) \oplus (I + E_{1,n})\big)
	=  C_{g} (I + O)  \oplus C_{\widetilde{g}} (I + E_{1,n})
	= C_{g} \oplus C_{h},
	\]
	where $C_h$ is the companion matrix of the monic polynomial
	\[
	h(x)=x^n+(a_0^{-1}a_1)x^{n-1}+\cdots+(a_0^{-1}a_{n-2})x^2+(a_0^{-1}a_{n-1}-1)x+a_0^{-1}.
	\]
	
	For $n\ge2$, the polynomials $\widetilde g(x)$ and $h(x)$ are distinct, since
	\[
	\widetilde g(x)=h(x) \iff a_0^{-1}a_{n-1}=a_0^{-1}a_{n-1}-1,
	\]
	which holds only if $1=0$, a contradiction. Hence $h(x)\neq\widetilde g(x)$. Furthermore, since $g(x)$ is not a power of a self-reciprocal polynomial, it is not a self-reciprocal polynomial. By Proposition~\ref{prop-rev-class}, it follows that
	$
	A(I+N)=C_g\oplus C_h\notin\cB_{2n}.
	$
\end{proof}

The following lemma is proved by an argument similar to that of Lemma~\ref{lem-special-nil-comp-type2} and concerns a
companion matrix in $\cB_n$ associated with a power of a self-reciprocal polynomial.

\begin{lemma}\label{lem-special-nil-comp-type1}
	Let $n \ge 3$ and let $f(x)$ be a monic polynomial of degree $n$ with nonzero constant term that is a power of a self-reciprocal irreducible polynomial. Suppose that $A=C_f\in \cB_n$, where $C_f$ is the companion matrix of $f(x)$. Then there exists a strictly upper triangular matrix $N$ such that $A(I+N)\notin \cB_n$.
\end{lemma}

\begin{proof}
	Let $f(x)=x^n+a_{n-1}x^{n-1}+\cdots+a_1x+a_0$, where $a_0\neq0$. 
	Since $f(x)$ is a power of a self-reciprocal polynomial, it is itself self-reciprocal. Hence $a_0^2=1$ and $a_{n-k}=a_0^{-1}a_k$ for $1\le k\le n-1$.
	
	Consider the strictly upper triangular matrix $N=E_{1,n}\in M_n(\IF)$, whose only nonzero entry is $1$ in the $(1,n)$ position. Then
	$A(I+N)=C_f(I+N)=C_h$, where $C_h$ is the companion matrix of the monic polynomial
	\[
	h(x)=x^n+a_{n-1}x^{n-1}+\cdots+a_2x^2+(a_1-1)x+a_0.
	\]
	
	Since $n \ge 3$ and $f(x)$ is self-reciprocal, we have $a_1=a_0^{-1}a_{n-1}$. Furthermore, since $1\neq0$, the equality $a_1-1=a_0^{-1}a_{n-1}$ cannot hold. Thus $h(x)$ is not self-reciprocal. By Proposition~\ref{prop-rev-class}, we conclude that $A(I+N)\notin\cB_n$.
\end{proof}

We note that when \(n=2\), the situation differs from that described in Lemma~\ref{lem-special-nil-comp-type1}. Every polynomial of the form \(f(x)=x^2+a_1x+1\) is self-reciprocal for all \(a_1\in\mathbb{F}\), and thus the above argument does not apply. For instance, if \(\mathrm{char}(\mathbb{F})\neq 2\) and \(f(x)=x^2+1\) be an irreducible self-reciprocal polynomial, then for any nilpotent matrix \(N\in M_2(\mathbb{F})\), the matrix \(C_f(I+N)\) has determinant \(1\), and hence \(C_f(I+N)\in\mathcal{B}_2\). Thus Lemma~\ref{lem-special-nil-comp-type1} fails for \(n=2\), and the case \(n=2\) must be treated separately. In particular, for \(A=C_f\in\mathcal{B}_2\) with \(f(0)=1\), we show that there exists a trace-zero matrix \(X\in\mathcal{B}_2\) such that \(AX\notin\mathcal{B}_2\); see also Remark~\ref{rem-decomposable-two}.

\begin{lemma}\label{lem-special-two}
	Let $\mathrm{char}(\IF)\neq 2$ and 
	$A=\begin{psmallmatrix}
		0 & -1\\
		1 & -a_1
	\end{psmallmatrix}\in\cB_{2}$, where $a_1\in\IF$. Then, for every nonzero $r\in\IF$, there exists a trace-zero matrix $X\in\cB_{2}$ such that $rAX\notin\cB_{2}$.
\end{lemma}

\begin{proof}
	Let $Y=\begin{psmallmatrix}1&y\\0&-1\end{psmallmatrix}\in\cB_{2}$, where $y\in\IF$ and $y\neq -a_1$. Then
	\[
	rAY=\begin{psmallmatrix}0&r\\ r&r(y+a_1)\end{psmallmatrix}.
	\]
	By Proposition~\ref{prop-rev-class}, $rAY$ belongs to $\cB_{2}$ if and only if its characteristic polynomial $\lambda^2-r(y+a_1)\lambda-r^2$ is self-reciprocal, which is equivalent to the conditions $r^4=1$ and $(r^2+1)(y+a_1)=0$.
	
	If $r^2\neq -1$, then $r^2+1\neq0$, and hence $y+a_1=0$, contradicting the choice of $y$. Thus $rAY\notin\cB_{2}$ in this case. Now suppose that $r^2=-1$ and set $Z=rY\in\cB_{2}$. Since the characteristic polynomial of $Z$ is $\lambda^2-1$, we have $Z\in\cB_{2}$. Moreover,
	\[
	rAZ=rA(rY)=-AY,
	\]
	whose characteristic polynomial is $\lambda^2+(y+a_1)\lambda-1$. Since $y+a_1\neq0$, this polynomial is not self-reciprocal, and therefore $rAZ\notin\cB_{2}$. Consequently, for $r^2\neq-1$ we may take $X=Y$, while for $r^2=-1$ we take $X=rY$. In either case, $rAX\notin\cB_{2}$, which completes the proof.
\end{proof}

\begin{remark}\label{rem-special-two-decomp}
	Observe that 
	\(\begin{psmallmatrix} a & 0 \\ 0 & a^{-1} \end{psmallmatrix}\) 
	is similar in \(M_2(\mathbb{F})\) to 
	\(\begin{psmallmatrix} 0 & -1 \\ 1 & a+a^{-1} \end{psmallmatrix}\),
	where \(a \neq 0,\pm 1\). Hence, if the rational canonical form of 
	\(A \in \mathcal{B}_n\) given in \eqref{eq-rev-class} contains the matrix  \(\begin{psmallmatrix} a & 0 \\ 0 & a^{-1} \end{psmallmatrix}\),
	arising as the direct sum of two \(1\times1\) blocks, then, up to similarity, it may be replaced by a companion matrix of the type considered in Lemma~\ref{lem-special-two}.
	\qed
\end{remark}

The following result considers a $3\times 3$ matrix $A\in\cB_3$ whose rational canonical form contains as diagonal blocks the matrices described in Lemma~\ref{lem-special-two} and Remark~\ref{rem-special-two-decomp} as diagonal blocks.

\begin{lemma}\label{lem-special-three}
	Let $\mathrm{char}(\IF) \neq 2$ and 
	$A=\begin{psmallmatrix}
		0& -1\\
		1& -a_1
	\end{psmallmatrix}\oplus \alpha \in \cB_{3}$,
	where $a_1\in\IF$ and $\alpha=\pm 1$. Then, for every nonzero $r\in\IF$, there exists a trace-zero matrix $X\in\cB_3$ such that $rAX\notin\cB_{3}$.
\end{lemma}

\begin{proof}
	Let
	$
	X_a =
	\begin{psmallmatrix}
		0&1&a\\
		0&0&-1\\
		1&a&0
	\end{psmallmatrix}\in M_3(\IF),
	$
	where $a\in\IF$. Since the characteristic polynomial of $X_a$ is $\lambda^3+1$, we have $X_a\in\cB_3$ for all $a\in\IF$. Note that
	$
	rAX_a =
	r\begin{psmallmatrix}
		0&0&1\\
		0&1&a+a_1\\
		\alpha&\alpha a&0
	\end{psmallmatrix},
	$
	where $r\neq 0$ and $\alpha=\pm 1$. The characteristic polynomial of $rAX_a$ is
	\[
	\det(\lambda I_3-rAX_a)
	=\lambda^{3}-r\lambda^{2}-(a^2+aa_1+1)(\alpha r^2)\lambda+\alpha r^3.
	\]
	This polynomial is self-reciprocal if and only if
	\[
	(\alpha r^3)^{-1}=\alpha r^3
	\quad\text{and}\quad
	(a^2+aa_1+1)(\alpha r^2)=\alpha r^4.
	\]
	Note that $(a^2+aa_1+1)(\alpha r^2)=\alpha r^4$ if and only if $a^2+aa_1+1-r^2=0$. Since $\mathrm{char}(\IF)\neq 2$, we have $|\IF|>2$. Hence, for any fixed $r\in\IF$, we can choose $a\in\IF$ such that $a^2+aa_1+1-r^2\neq 0$. For this choice of $a$, let $X:=X_a$. Then the characteristic polynomial of $rAX$ is non-self-reciprocal. In view of Proposition~\ref{prop-rev-class}, we conclude that $rAX\notin\cB_3$ for all nonzero $r\in\IF$.
\end{proof}

We note the final observation needed to prove Lemma~\ref{lem-key-nonunital-B}. 
The rational canonical form of \(A \in \mathcal{B}_n\) may contain 
\((1)\oplus(-1)\) as a diagonal block; for example, this occurs if 
\(A\) is similar to \(I_k \oplus (-I_{n-k})\), where \(k \ge 1\), \(n-k \ge 1\), 
and \(I_\ell\) denotes the identity matrix in \(M_\ell(\mathbb{F})\). 
Below we describe a technique for handling such situations.

\begin{remark}\label{rem-decomposable-two}
	Let \(\mathrm{char}(\mathbb{F})\neq 2\). The matrix 
	\(\begin{psmallmatrix}1&0\\0&-1\end{psmallmatrix}\) is similar to the companion matrix 
	\(C_f=\begin{psmallmatrix}0&1\\1&0\end{psmallmatrix}\in\mathcal{B}_2\), where \(f(x)=x^2-1\). 
	Thus, if the rational canonical form of \(A\in\mathcal{B}_n\) given in \eqref{eq-rev-class} 
	contains \(\begin{psmallmatrix}1&0\\0&-1\end{psmallmatrix}\) as a diagonal block, then, up to similarity, 
	this block may be replaced by \(C_f\). Moreover, for 
	\(N=E_{1,2}=\begin{psmallmatrix}0&1\\0&0\end{psmallmatrix}\in M_2(\mathbb{F})\), 
	a direct computation gives \(C_f(I+N)=\begin{psmallmatrix}0&1\\1&1\end{psmallmatrix}\notin\mathcal{B}_2\).
\end{remark}

We observe that $\cB_n\subseteq \cC_n\subseteq \cD_n$. These inclusions are generally proper; see \cite[Fact~2, Corollary~4]{Ba}, \cite[p.~161]{GHR}, \cite[Theorem~2]{HK}, and Remark~\ref{rem-order-3-char-2} for examples. Although a complete characterization of $\cC_n$ is not known, the cases in which $\cC_n=\cD_n$ were completely classified by Ballantine in \cite[Fact~5]{Ba}.
We recall this result below.

\begin{remark}[\cite{Ba}]\label{rem_C_n}
	We have $\mathcal{C} _ {n} =\mathcal{D} _ {n} $ if and only if at least one of the following holds:
	\begin{enumerate}[label={\normalfont(\roman*)}]
		\item $n\le 2$;
		\item $|\IF|=2,3,$ or $5$;
		\item $n=3$ and either $\operatorname{char}(\IF)=3$ or $t^2+t+1$ is irreducible in $\IF[t]$;
		\item $n=4$ and $\operatorname{char}(\IF)=2$.
	\end{enumerate}
\end{remark}

Let $\mathfrak{sl}_n(\IF)$ denote the subspace of $M_n(\IF)$ consisting of trace-zero matrices. Note that $\mathfrak{sl}_n(\IF)$ is the subspace spanned by $\mathcal{N}$. We conclude this section by recalling the linear preservers of the set $\mathcal{N}$, as characterized by Botta et al.\ in \cite{BPW}. Their main result builds on Gerstenhaber’s work in \cite{Ge}, which assumes that the underlying field contains at least $n$ elements. Serežkin later removed this cardinality restriction in \cite{Se}, showing that Gerstenhaber’s result holds over arbitrary fields. For more recent and simpler proofs of both Gerstenhaber’s theorem and Serežkin’s generalization, see \cite{MOR}. The following characterization follows from the results in \cite{BPW} and \cite{Se}.

\begin{proposition}[\cite{BPW,Se}]\label{prop-nil-preserver}
	Let $\IF$ be a field, and let $T : M_n(\IF)\rightarrow M_n(\IF)$ be a bijective linear map with $T(\mathcal{N}) \subseteq \mathcal{N}$. Then there exist a nonzero scalar $r \in \IF$ and an invertible matrix $P \in M_n(\IF)$ such that the restriction of $T$ to $\mathfrak{sl}_n(\IF)$ is one of the following:
	$$X\mapsto rPXP^{-1} \quad\text{or}\quad X\mapsto rPX^{t}P^{-1}.$$
\end{proposition}

\section{Proof of the main theorem}\label{proofs}

In this section, we prove Theorem~\ref{thm-main}. We begin by applying the results of Section~\ref{section2} to show that a bijective linear map $T: M_n(\IF)\rightarrow M_n(\IF)$ preserving $\cB_n$, under suitable additional assumptions, must satisfy $T(I)=\pm I$; see also Lemma~\ref{lem-main-all-form}. This property plays a crucial role in the proof of Theorem~\ref{thm-main}(i).

\begin{lemma}\label{lem-key-nonunital-B}
	Let $n\ge 2$, $\mathrm{char}(\IF)\neq 2$, and let $T: M_n(\IF)\rightarrow M_n(\IF)$ be a bijective linear map such that $T(\cB_n)\subseteq \cB_n$. Suppose that there exist a nonzero scalar $r\in\IF$ and an invertible matrix $P\in M_n(\IF)$ such that the restriction of $T$ to $\mathfrak{sl}_n(\IF)$ has one of the following forms:
	\begin{equation*}
		X \mapsto r\,T(I)PXP^{-1} \quad \text{or} \quad X \mapsto r\,T(I)PX^tP^{-1}.
	\end{equation*}
	Then $T(I)=\pm I$.
\end{lemma}

\begin{proof}
	We consider only the first case, as the second follows by a similar argument. Let $A=T(I)\in\cB_n$. Suppose that $A\neq \pm I$. Without loss of generality, we may assume that $A\in\cB_n$ is in the rational canonical form described in Proposition~\ref{prop-rev-class}. Note that for every invertible matrix $P\in M_n(\IF)$ we have $P^{-1}\cB_nP\subseteq \cB_n$ and $P^{-1}\mathfrak{sl}_n(\IF)P\subseteq \mathfrak{sl}_n(\IF)$. Hence, without loss of generality, we may assume that $P=I$; that is, $T(X)=rAX$ for all $X\in\mathfrak{sl}_n(\IF)$, where $A\in\cB_n$ is in rational canonical form. Let $m$ be the maximum size of the indecomposable companion matrix appearing in the rational canonical form of $A$ given in~\eqref{eq-rev-class}. In the following cases, we construct a matrix $X\in\cB_n$ such that $T(X)\notin\cB_n$, which yields a contradiction.
	
	\textbf{Case (a).} Let $m\ge 3$. Then, up to similarity,
	\begin{equation*}
		A=A_1\oplus A_2,
	\end{equation*}
	where $A_2\in M_{n-m}(\IF)$ and $A_1=C_f$ or $A_1=C_g\oplus C_{\widetilde g}$, where $C_f$ and $C_g$ are companion matrices as considered in Lemma~\ref{lem-special-nil-comp-type1} and Lemma~\ref{lem-special-nil-comp-type2}, respectively. Then there exists a strictly upper triangular matrix $N_1$ such that $A_1(I+N_1)\notin\cB_m$. Consider the strictly upper triangular matrix $N=N_1\oplus O$, where $O\in M_{n-m}(\IF)$ is the zero matrix of order $n-m$. Then for every nonzero $r\in\IF$, the matrix $X=I+r^{-1}N$ is a unipotent matrix in $\cB_n$ such that
	$$T(X)=T(I+r^{-1}N)=A(I+N)=A_1(I+N_1)\oplus(A_2(I+O))=A_1(I+N_1)\oplus A_2.$$
	Since $A_1(I+N_1)\notin\cB_m$ and $A_2\in\cB_{n-m}$, it follows that $T(X)\notin\cB_n$.
	
	\textbf{Case (b).} Let $m=2$. Up to similarity, we may write
	\begin{equation*}
		A = A_1 \oplus A_2,
	\end{equation*}
	where $A_1 = C_g \oplus C_{\widetilde g}$ and $A_2 \in M_{n-4}(\IF)$, with $C_g \in M_2(\IF)$ the companion matrix considered in Lemma~\ref{lem-special-nil-comp-type2}. Using an argument similar to that in Case~(a), we can construct a unipotent matrix $X \in \cB_n$ such that $T(X) \notin \cB_n$.
	
	\textbf{Case (c).}	Let \(m\le 2\), and suppose that the rational canonical form of \(A\) given in~\eqref{eq-rev-class} contains 
	\(\begin{psmallmatrix}1&0\\0&-1\end{psmallmatrix}\) as a diagonal block,
	arising as the direct sum of two \(1\times1\) blocks. Since \(\mathbb{F}\neq \mathbb{Z}_2\), by Remark~\ref{rem-decomposable-two}, up to similarity we may write
	\[
	A=A_1\oplus A_2,
	\]
	where \(A_1=\begin{psmallmatrix}0&1\\1&0\end{psmallmatrix}\in M_2(\mathbb{F})\) and \(A_2\in M_{n-2}(\mathbb{F})\). By Remark~\ref{rem-decomposable-two}, there exists a strictly upper triangular matrix \(N_1\) such that \(A_1(I+N_1)\notin\mathcal{B}_2\). Using an argument similar to that in Case~(a), we obtain a unipotent matrix \(X\in\mathcal{B}_n\) such that \(T(X)\notin\mathcal{B}_n\).
	
	\textbf{Case (d).} Let $m=2$, and suppose that $A$ is a direct sum of matrices of the form
	\begin{equation*}
		\begin{psmallmatrix}a&0\\0&a^{-1}\end{psmallmatrix},\quad
		\begin{psmallmatrix}0&-1\\1&-b\end{psmallmatrix},\quad
		I_2,
	\end{equation*}
	or of the form
	\begin{equation*}
		\begin{psmallmatrix}a&0\\0&a^{-1}\end{psmallmatrix},\quad
		\begin{psmallmatrix}0&-1\\1&-b\end{psmallmatrix},\quad
		-I_2,
	\end{equation*}
	where $a,b\in\IF$ with $a\neq0,\pm1$, and $I_2$ is the identity matrix in $M_2(\IF)$. By Remark~\ref{rem-special-two-decomp}, up to similarity we may write $A=A_1\oplus A_2\oplus A_3$, where
	$$A_1=\begin{psmallmatrix}0&-1\\1&-a_1\end{psmallmatrix},\qquad
	A_2=\bigoplus_{i=1}^s\begin{psmallmatrix}0&-1\\1&-b_i\end{psmallmatrix},\qquad
	A_3=\alpha I=\bigoplus_{i=1}^k\begin{psmallmatrix}\alpha&\\&\alpha\end{psmallmatrix},$$
	with $\alpha=\pm1$, $s=(n-2-2k)/2$, $a_1\in\IF$, and each $b_i\in\IF$. By Lemma~\ref{lem-special-two}, there exists a trace-zero matrix $X_1\in\cB_2$ such that $rA_1X_1\notin\cB_2$ for every nonzero $r$. Let
	$X_2=\bigoplus_{i=1}^s\begin{psmallmatrix}0&1\\1&0\end{psmallmatrix}\in\cB_{2s}$ and $X_3=\bigoplus_{i=1}^k\begin{psmallmatrix}-\alpha&0\\0&\alpha\end{psmallmatrix}\in\cB_{2k}$. Then for every nonzero $r$ we have
	$rA_2X_2=\bigoplus_{i=1}^s\begin{psmallmatrix}-r&0\\rb_i&r\end{psmallmatrix}$,
	which is similar to $\bigoplus_{i=1}^s\begin{psmallmatrix}-r&\\&r\end{psmallmatrix}$, and
	$rA_3X_3=\bigoplus_{i=1}^k\begin{psmallmatrix}-r&\\&r\end{psmallmatrix}$.
	Consider the trace-zero matrix $X=X_1\oplus X_2\oplus X_3\in\cB_n$. Then $T(X)=rAX=(rA_1X_1)\oplus(rA_2X_2)\oplus(rA_3X_3)$, and by Proposition~\ref{prop-rev-class} we conclude that $T(X)\notin\cB_n$.
	
	\textbf{Case (e).} Let $m=2$, and suppose that, up to similarity, $A$ has the form
	$$A=A_1\oplus A_2\oplus A_3,$$
	where $A_1=\begin{psmallmatrix}0&-1\\1&-a_1\end{psmallmatrix}\oplus\alpha$, $A_2=\bigoplus_{i=1}^s\begin{psmallmatrix}0&-1\\1&-b_i\end{psmallmatrix}\in M_{n-3-2k}(\IF)$, and $A_3=\bigoplus_{i=1}^k\begin{psmallmatrix}\alpha&\\&\alpha\end{psmallmatrix}$, with $s=(n-3-2k)/2$, $\alpha=\pm1$, $a_1\in\IF$, and each $b_i\in\IF$. By Lemma~\ref{lem-special-three}, there exists a trace-zero matrix $X_1\in\cB_3$ such that $rA_1X_1\notin\cB_3$ for every nonzero $r$. Let $X_2=\bigoplus_{i=1}^s\begin{psmallmatrix}0&1\\1&0\end{psmallmatrix}\in\cB_{2s}$ and $X_3=\bigoplus_{i=1}^k\begin{psmallmatrix}-1&0\\0&1\end{psmallmatrix}\in\cB_{2k}$. Consider the trace-zero matrix $X=X_1\oplus X_2\oplus X_3\in\cB_n$. Using an argument similar to that in Case~(d), we obtain $T(X)\notin\cB_n$.
	
	By Proposition~\ref{prop-rev-class}, if $A\neq\pm I$, then $A\in\cB_n$ is similar to one of the five cases described above. Therefore, if $A\neq\pm I$, we can construct a unipotent or trace-zero matrix $X\in\cB_n$ such that $T(X)\notin\cB_n$, contradicting the assumption that $T(\cB_n)\subseteq\cB_n$. Hence $T(I)=A=\pm I$, which completes the proof.
\end{proof}

Recall that $\cS_n=\cB_n$, $\cC_n$, or $\cD_n$. Moreover, $I+tN\in\cB_n\subseteq\cS_n$ for all $t\in\IF$ and $N\in\mathcal N$. We now observe the following important fact: if $T$ is a bijective linear map from $M_n(\IF)$ to $M_n(\IF)$ satisfying $T(\cS_n)\subseteq\cS_n$, then
$$T(I)^{-1}T(\mathcal N)\subseteq\mathcal N.$$

To prove this, we use the following fundamental result, which characterizes a nilpotent matrix $N$ in terms of the determinant of the matrix pencil defined by $I$ and $N$.

\begin{proposition}\label{lem-nil-det-criteria}
	Let $\IF$ be a field. Then $N\in M_n(\IF)$ is nilpotent if and only if $\det(I+tN)=1$ for all $t\in\IF$.
\end{proposition}

\begin{proof}
	Since every nilpotent matrix $N$ is similar to a strictly upper triangular matrix in $M_n(\IF)$, we have $\det(I+tN)=1$ for all $t\in\IF$.
	
	Conversely, suppose that $\det(I+tN)=1$ for all $t\in\IF$. Write
	\[
	\det(I+tN)=c_nt^n+c_{n-1}t^{n-1}+\dots+c_1t+c_0,
	\]
	where $c_k=E_k(\lambda_1,\dots,\lambda_n)$ denotes the $k$th elementary symmetric function of the eigenvalues $\lambda_1,\dots,\lambda_n$ of $N$. Since $\det(I+tN)-1$ is the zero polynomial for all $t\in\IF$, we obtain $c_0=1$ and $c_i=0$ for all $1\le i\le n$. Hence
	\[
	\det(\lambda I-N)
	=\lambda^{n}\det(I-\lambda^{-1}N)
	=c_0\lambda^n-c_1\lambda^{n-1}+c_2\lambda^{n-2}+\dots+(-1)^nc_n
	=\lambda^n.
	\]
	By the Cayley--Hamilton theorem, it follows that $N^n=0$. Hence $N$ is nilpotent.
\end{proof}

Along the same lines as the proof of Proposition~\ref{lem-nil-det-criteria}, we note the following observation.

\begin{remark}\label{rem-nil-neg-det-criteria}
	Let $N\in M_n(\IF)$ such that $\det(I+tN)\in\{-1,1\}$ for all $t\in\IF$. Since $(\det(I+tN))^2-1$ is the zero polynomial for all $t\in\IF$, the same argument as in Proposition~\ref{lem-nil-det-criteria} shows that $N$ is nilpotent.
\end{remark}

The following result is an immediate consequence of Remark~\ref{rem-nil-neg-det-criteria}.

\begin{corollary}\label{new_cor}
	Let $\IF$ be a field, and let $T:M_n(\IF)\rightarrow M_n(\IF)$ be a bijective linear map such that $T(\cS_n)\subseteq\cS_n$. Then $T(I)^{-1}T(N)$ is nilpotent for every nilpotent matrix $N$.
\end{corollary}

\begin{proof}
	Let $A=T(I)$. Since $A\in\cS_n$, we have $\det(A)\in\{-1,1\}$. If $N\in\mathcal N$, then $I+tN\in\cS_n$ for all $t\in\IF$, and hence
	\[
	T(I+tN)=T(I)+tT(N)=A+tT(N)\in\cS_n.
	\]
	Thus $\det(A+tT(N))\in\{-1,1\}$ for all $t\in\IF$. Since
	$
	\det(A+tT(N))=\det(A)\det(I+tA^{-1}T(N)),
	$
	and $\det(A)\in\{-1,1\}$, we obtain $\det(I+tA^{-1}T(N))\in\{-1,1\}$ for all $t\in\IF$. By Remark~\ref{rem-nil-neg-det-criteria}, it follows that $A^{-1}T(N)$ is nilpotent.
\end{proof}

Recall that the subspace $\mathfrak{sl}_n(\IF)$ of $M_n(\IF)$ is spanned by the set $\mathcal{N}$ of nilpotent matrices. The following result describes the restriction of bijective linear preservers of $\cS_n$ to $\mathfrak{sl}_n(\IF)$.

\begin{lemma}\label{lem-main-all-form}
	Let $n\ge 2$ and let $T: M_n(\IF) \rightarrow M_n(\IF)$ be a bijective linear map such that $T(\cS_{n}) \subseteq \cS_{n}$. Then there exist a scalar $r \in \IF$ with $r^{2n}=1$ and an invertible matrix $P\in M_n(\IF)$ such that the restriction of $T$ to $\mathfrak{sl}_n(\IF)$ has one of the following forms:
	\begin{equation*}
		X \mapsto r T(I)PXP^{-1} \quad \text{or} \quad X \mapsto r T(I)PX^tP^{-1}.
	\end{equation*}
	Furthermore, if $n$ is a multiple of the characteristic of $\IF$, then $r=1$.
\end{lemma}

\begin{proof}
	Let $T: M_n(\IF) \rightarrow M_n(\IF)$ be a bijective linear map satisfying $T(\cS_{n}) \subseteq \cS_{n}$, where $n\ge 2$. Let $A:=T(I)\in\cS_n$. Using Corollary~\ref{new_cor}, we obtain $A^{-1}T(N)\in\mathcal{N}$ for all $N\in\mathcal{N}$. Define the map $\phi : M_n(\IF) \rightarrow M_n(\IF)$ by
	\begin{equation*}
		\phi(X)=A^{-1}T(X)\quad \text{for all } X\in M_n(\IF).
	\end{equation*}
	Then $\phi$ is a bijective linear map such that $\phi(I)=I$ and $\phi(\mathcal{N})\subseteq\mathcal{N}$. Since the subspace $\mathfrak{sl}_n(\IF)$ of trace-zero matrices in $M_n(\IF)$ is spanned by the set $\mathcal{N}$ of nilpotent matrices, it follows that $\phi(\mathfrak{sl}_n(\IF))\subseteq\mathfrak{sl}_n(\IF)$.
	
	In view of Proposition~\ref{prop-nil-preserver}, there exist a nonzero scalar $r\in\IF$ and an invertible matrix $P\in M_n(\IF)$ such that the restriction of $\phi$ to $\mathfrak{sl}_n(\IF)$ has one of the following forms:
	$$X\mapsto rPXP^{-1} \quad \text{or} \quad X\mapsto rPX^{t}P^{-1}.$$
	Therefore, the restriction of $T$ to $\mathfrak{sl}_n(\IF)$ has one of the following forms:
	\begin{equation*}
		X \mapsto rAPXP^{-1} \quad \text{or} \quad X \mapsto rAPX^tP^{-1}.
	\end{equation*}
	
	Now consider $X = E_{1,2} + \dots + E_{n-1,n} + E_{n,1} \in \mathfrak{sl}_n(\IF)$. Then $X \in \cB_{n} \subseteq \cS_n$ and $\det(X)=(-1)^{n-1}$. Since $T(X)\in\cS_{n}\subseteq\cD_{n}$, we have
	$$\det(T(X))=r^n\det(A)\det(X)=\pm1.$$
	Using $\det(A)=\pm1$, we obtain $r^n=\pm1$, that is, $r^{2n}=1$.
	
	If $n$ is a multiple of $\mathrm{char}(\IF)$, then the identity matrix $I\in\cS_n$ has trace zero. Hence $T(I)=rT(I)$, which implies $r=1$. This completes the proof.
\end{proof}

We now determine the scalar $r$ for the bijective linear preservers of $\cS_n$ described in Lemma~\ref{lem-main-all-form}.

\begin{lemma}\label{lem-main-all-unital-scalar}
	Let $n \ge 2$, $\mathrm{char}(\IF) \neq 2$, and let $T: M_n(\IF) \rightarrow M_n(\IF)$ be a bijective linear map such that $T(\cS_{n}) \subseteq \cS_{n}$. Suppose that there exist a nonzero scalar $r \in \IF$ and an invertible matrix $P \in M_n(\IF)$ such that the restriction of $T$ to $\mathfrak{sl}_n(\IF)$ has one of the following forms:
	\begin{equation*}
		X \mapsto r T(I)PXP^{-1} \quad \text{or} \quad X \mapsto r T(I)PX^tP^{-1}.
	\end{equation*}
	Then
	$$
	r=
	\begin{cases}
		\pm 1 & \text{if } n=2, \\
		1 & \text{if } n \ge 3 .
	\end{cases}
	$$
\end{lemma}

\begin{proof}
	Let $A:=T(I) \in \cS_n \subseteq \cD_n$. We consider the case $T(X)=rAPXP^{-1}$ for all $X\in \mathfrak{sl}_n(\IF)$, since the second case can be handled using a similar argument. Furthermore, by replacing $T$ with the map $X \mapsto T(P^{-1}XP)$ on $M_n(\IF)$ if necessary, we may assume that $T(I)=A$ and $T(X)=rAX$ for all $X\in\mathfrak{sl}_n(\IF)$.
	
	If $n$ is a multiple of the characteristic $\mathrm{char}(\IF)$, the result follows from Lemma~\ref{lem-main-all-form}. Now suppose that $n$ is not a multiple of $\mathrm{char}(\IF)$. To complete the proof, we consider the following cases depending on $n$.
	
	\textbf{Case (a).} Let $n=2$. Consider the trace-zero matrix
	$X = \begin{psmallmatrix} 1 & \\ & -1 \end{psmallmatrix} \in \cB_{2} \subseteq \cS_{2}$. Then
	$T(X)=A \begin{psmallmatrix} r & \\ & -r \end{psmallmatrix} \in \cS_{2} \subseteq \cD_{2}$.
	Since $A \in \cD_{2}$, it follows that
	$\begin{psmallmatrix} r & \\ & -r \end{psmallmatrix} \in \cD_{2}$, that is, $r^2 = \pm1$.
	
	Suppose that $r^2 = -1$. Consider
	$Y =  \begin{psmallmatrix} 0 & 1 \\ -1 & 1 \end{psmallmatrix} \in \cB_{2} \subseteq \cS_{2}$.
	Since $\mathrm{char}(\IF) \neq 2$, we can write
	$$
	Y= \frac{1}{2}I + \begin{psmallmatrix} -\frac{1}{2} & 1 \\ -1 & \frac{1}{2} \end{psmallmatrix}.
	$$
	This implies that
	$$
	T(Y)= \frac{1}{2} A + rA\begin{psmallmatrix} -\frac{1}{2} & 1 \\ -1 & \frac{1}{2} \end{psmallmatrix}
	= A \begin{psmallmatrix} \frac{1-r}{2} & r \\ -r & \frac{1+r}{2} \end{psmallmatrix}.
	$$
	Since $ r^2=-1$, we obtain
	$\det(T(Y)) = \frac{3r^2+1}{4} \det(A)= -\frac{1}{2}\det(A)$.
	Since $A \in  \cD_{2}$, we have $\det(T(Y)) = \pm \frac{1}{2}$.
	If $\mathrm{char}(\IF) \neq 3$, then $\det(T(Y)) \neq \pm1$, contradicting the fact that $T(Y) \in \cS_{2} \subseteq \cD_{2}$.
	
	Now assume that $\mathrm{char}(\IF) =3$. Consider
	$Z =  \begin{psmallmatrix} 0 & 1 \\ -1 & r \end{psmallmatrix} \in \cB_{2} \subseteq \cS_{2}$,
	where $r\in \IF$ satisfies $r^2=-1$. Then
	$$
	Z= \frac{r}{2}I + \begin{psmallmatrix} -\frac{r}{2} & 1 \\ -1 & \frac{r}{2} \end{psmallmatrix}.
	$$
	This implies that
	$$
	T(Z)
	=  A\begin{psmallmatrix} \frac{r-r^2}{2} & r \\ -r & \frac{r+r^2}{2} \end{psmallmatrix}
	= A\begin{psmallmatrix} \frac{r+1}{2} & r \\ -r & \frac{r-1}{2} \end{psmallmatrix}
	\in  \cS_{2} \subseteq \cD_{2}.
	$$
	Since $ r^2=-1$ and $\mathrm{char}(\IF) =3$, we obtain
	$\det(T(Z)) =  \frac{5r^2-1}{4} \det(A) =  \frac{-6}{4}\det(A) =0$, a contradiction.
	
	Therefore, $r^2 \neq -1$, which implies that $r=\pm 1$.
	
	\textbf{Case (b).} Let $n\ge3$. To prove that $r=1$, we consider two subcases depending on whether $\cS_{n}=\cC_n$ or $\cD_n$, or $\cS_{n}=\cB_n$.
	
	\textbf{Subcase (i).} Let $T(\cC_n)\subseteq\cC_n$ or $T(\cD_n)\subseteq\cD_n$. Suppose that $r \neq 1$.
	Let $X = X_0 \oplus O_{n-2}$ and $Y = Y_0 \oplus O_{n-3}$ be matrices in $\mathfrak{sl}_n(\IF)$, where
	$$
	X_0 =
	\begin{psmallmatrix} 0&2\\ 1&0 \end{psmallmatrix}
	\in M_{2}(\IF)
	\quad\text{and}\quad
	Y_0 =
	\begin{psmallmatrix}-1& \frac{1}{2}& 0\\
		-1& 0& 0\\
		0& 0& 1
	\end{psmallmatrix}
	\in M_{3}(\IF).
	$$
	Then $\det(I+X) = -1$ and $\det(I+Y) = 1$. Moreover,
	$T(I)+T(X) = A + rAX$ and $T(I)+T(Y) = A + rAY$.
	Thus
	$$
	\alpha:= \det(T(I + X)) = (1-2r^2) \det(A) 	\quad\text{and}\quad
	\beta:= \det (T(I + Y))
	=\left(1+\frac{r^2(r-1)}{2}\right)\det(A).
	$$
	
	Recall that \cite[Corollary 2]{Ba} implies $\cC_2=\cD_2$, so $I_2+X_0 \in \cC_2$. Moreover, by \cite[Corollary 3(b)]{Ba}, every non-scalar matrix in $M_{3}(\IF)$ belongs to $\cC_3$ if and only if it belongs to $\cD_3$. Since $I_3+Y_0$ is a non-scalar matrix in $\cD_3$, it follows that $I_3+Y_0$ belongs to $\cC_3$. Hence $I+X=(I_2+X_0)\oplus I_{n-2}$ and $I+Y=(I_3+Y_0)\oplus I_{n-3}$ belong to $\cC_n$, and therefore also to $\cD_n$.
	
	Since $I+X,I+Y\in\cC_{n}\subseteq\cD_n$, we have $T(I+X),T(I+Y)\in\cD_n$. This implies that $\alpha,\beta\in\{-1,1\}$. Since $A\in\cD_n$, we have $\det(A)\in\{-1,1\}$. Using $\alpha=(1-2r^2)\det(A)\in\{-1,1\}$ and $r\neq0,1$, we obtain $r=-1$. For $r=-1$, we have $\beta=0$, which contradicts $\beta\in\{\pm1\}$. Hence $r=1$.
	
	\textbf{Subcase (ii).} Let $T(\cB_n)\subseteq\cB_n$. By Lemma~\ref{lem-key-nonunital-B}, we have $A=T(I)=\pm I$. Hence, replacing $T$ with $-T$ if necessary, we may assume that $T(I)=I$ and $T(X)=rX$ for all $X\in\mathfrak{sl}_n(\IF)$.
	
	Now consider
	$X =  \begin{psmallmatrix} I_{n-1} & \\ & -1 \end{psmallmatrix} \in \cB_{n}$.
	Then the trace of $X$ equals $n-2$.
	
	Suppose that $n-2$ is not a multiple of $\mathrm{char}(\IF)$. Then we can write
	$$
	X=\frac{n-2}{n} I_{n} + \begin{psmallmatrix}
		\frac{2}{n}I_{n-1} & \\
		& \frac{2(-n+1)}{n}
	\end{psmallmatrix},
	$$
	where the second matrix belongs to $\mathfrak{sl}_n(\IF)$.
	This implies
	\begin{equation*}
		T(X)=\frac{n-2}{n} I_n + r
		\begin{psmallmatrix}
			\frac{2}{n}I_{n-1} & \\
			& \frac{2(-n+1)}{n}
		\end{psmallmatrix}
		=
		\begin{psmallmatrix}
			\frac{n-2+2r}{n}I_{n-1} & \\
			& \frac{n-2+2r(-n+1)}{n}
		\end{psmallmatrix}.
	\end{equation*}
	
	Let $\alpha=\frac{n-2+2r}{n}$. Since $T(X)\in\cB_{n}$ and $n\ge3$, we have
	$\alpha,\alpha-2r\in\{-1,1\}$.
	This implies $\alpha-2r=\pm\alpha$, hence $r=0$ or $r=\alpha$. Since $r\neq0$, we must have $r=\alpha$. Thus
	$$
	r=\alpha=\frac{n-2+2r}{n}
	\implies n(r-1)=2(r-1).
	$$
	Since $n\ge3$ and $n-2$ is not a multiple of $\mathrm{char}(\IF)$, we conclude that $r=1$.
	
	Now suppose that $n-2$ is a multiple of $\mathrm{char}(\IF)$. Then $X$ is trace-zero and $T(X)=rX\in\cB_n$. Since $n\ge3$, we have $r=\pm1$. Suppose $r=-1$. Consider
	$Y =  \begin{psmallmatrix} I_{n-3} & \\ & -I_{3} \end{psmallmatrix} \in \cB_{n}$.
	Since $\mathrm{char}(\IF)\neq2$ and $n-2$ is a multiple of $\mathrm{char}(\IF)$, the trace of $Y$ equals $n-6\neq0$. Hence
	$$
	Y=\frac{n-6}{n} I_n +
	\begin{psmallmatrix}
		\frac{6}{n}I_{n-3} & \\
		& \frac{2(-n+3)}{n}I_3
	\end{psmallmatrix}.
	$$
	This implies
	$$
	T(Y)=\frac{n-6}{n}I_n-
	\begin{psmallmatrix}
		\frac{6}{n}I_{n-3} & \\
		& \frac{2(-n+3)}{n}I_3
	\end{psmallmatrix}
	=
	\begin{psmallmatrix}
		\frac{n-12}{n}I_{n-3} & \\
		& \frac{3n-12}{n}I_3
	\end{psmallmatrix}.
	$$
	
	If $\mathrm{char}(\IF)=3$, then $\det(T(Y))=0$, a contradiction.
	Suppose $\mathrm{char}(\IF)\neq 3$. Since $n-2$ is a multiple of $\mathrm{char}(\IF)$, the eigenvalues of $T(Y)$ are $-5$ and $-3$. But $T(Y)\in\cB_n$, so its eigenvalues must lie in $\{-1,1\}$, which is impossible when $\mathrm{char}(\IF)\neq 2,3$. Hence $r=1$. This completes the proof.
\end{proof}

If $n$ is not a multiple of $\mathrm{char}(\IF)$, then every matrix $Z \in M_n(\IF)$ admits a decomposition of the form $Z = aI + X$, where $a \in \IF$ and $X \in \mathfrak{sl}_n(\IF)$. In this case, the proof of Theorem~\ref{thm-main} follows directly from Lemmas~\ref{lem-key-nonunital-B}, \ref{lem-main-all-form}, and \ref{lem-main-all-unital-scalar}. However, if $n$ is a multiple of $\mathrm{char}(\IF)$, then the identity matrix $I$ also belongs to $\mathfrak{sl}_n(\IF)$, and such a decomposition is no longer possible. To fully determine the linear map $T$ in this case, it is necessary to know its action on a matrix with nonzero trace, in addition to its behavior on $\mathfrak{sl}_n(\IF)$. The following result will be useful for this purpose. Since the case $n = 2$ does not occur under our assumptions that $\mathrm{char}(\IF) \neq 2$ and that $n$ is a multiple of $\mathrm{char}(\IF)$, we consider $n \ge 3$ here; see also Lemma~\ref{lem-char-two}.

\begin{lemma}\label{lem-for-n-divides-char}
	Let $n \geq 3$, $\mathrm{char}(\IF) \neq 2$, and suppose that $n$ is a multiple of $\mathrm{char}(\IF)$. Let $T: M_n(\IF) \rightarrow M_n(\IF)$ be a bijective linear map such that $T(\cS_n) \subseteq \cS_n$. Suppose that there exists an invertible matrix $P \in M_n(\IF)$ such that the restriction of $T$ to $\mathfrak{sl}_n(\IF)$ has one of the following forms:
	\begin{equation*}
		X \mapsto T(I)PXP^{-1} \quad \text{or} \quad X \mapsto T(I)PX^tP^{-1}.
	\end{equation*}
	Then $T(E_{i,i}) = T(I)PE_{i,i}P^{-1}$ for all $1 \leq i \leq n$.
\end{lemma}

\begin{proof}
	We may assume that $T(I)=I$ and $T(X)=X$ for all $X \in \mathfrak{sl}_n(\IF)$, since the remaining cases can be handled using a similar argument. Note that to prove the result, it suffices to show that $T(E_{i,i})=E_{i,i}$ for all $1 \leq i \leq n$.
	
	Now suppose that $T(E_{i,i}) = E_{i,i} + Y_i$ for all $1\leq i \leq n$. Since $T(E_{j,j}-E_{k,k}) = E_{j,j}-E_{k,k}$, we have $Y_j = Y_k$ for all $1 \leq j,k \leq n$. Therefore,
	\begin{equation*}
		T(E_{i,i}) = E_{i,i} + Y \quad \text{for all } 1\leq i \leq n,
	\end{equation*}
	where $Y = (y_{s,t}) \in M_n(\IF)$. We now prove that $Y=0$ in the following three steps.
	
	\textbf{Step 1.}
	First, we show that the off-diagonal entries of $Y$ are zero, i.e., $y_{i,j}=0$ for all $i \neq j$. Suppose that $y_{s,t} \neq 0$ for some $s \neq t$. Without loss of generality, we may replace $Y$ with $PYP^{-1}$ for a suitable permutation matrix $P \in M_n(\IF)$ and assume that $(s,t) = (1,2)$. Thus, it suffices to prove that $y_{1,2}=0$.
	
	Let $n = 2k+1\geq 3$ be odd. Consider $X = E_{1,1} + X_1$, where
	$$
	X_1 =  \alpha E_{2,1}  - \sum_{i> j} y_{i,j} E_{i,j} + \sum_{i=2}^n (-1)^i E_{i,i}, \qquad \alpha \in \IF.
	$$
	Then $X \in \cB_n\subseteq \cS_n$ is a lower triangular matrix with diagonal entries $(1,1,-1,  \dots, 1,-1)$. Since $T(X) = T(E_{1,1})+X_1 = X +Y$, it follows that
	$$
	T(X) =\begin{psmallmatrix} Z_{11} & Z_{12} \cr 0 & Z_{22}\cr\end{psmallmatrix},
	$$
	where $Z_{11} = \begin{psmallmatrix} 1+y_{1,1} & y_{1,2} \cr \alpha & 1 + y_{2,2}\cr\end{psmallmatrix}$
	and $Z_{22}$ is an upper triangular matrix with constant diagonal entries. Thus, $\det(T(X))$ is a linear function of $\alpha\in\IF$, and
	$$
	\det(T(X)) =   \det(Z_{11})\det(Z_{22}) 
	= (- \alpha y_{1,2} + (1+y_{1,1})(1+y_{2,2}))\det(Z_{22}),
	$$
	where $\det(Z_{22}) = \prod_{i=3}^n (y_{i,i} + (-1)^i)$ is independent of $\alpha$. Moreover, since $X \in \cB_n \subseteq \cS_n$, we have $T(X) \in \cS_n\subseteq\cD_n$, and hence $(\det(T(X)))^2=1$. Therefore $(\det(T(X)))^2=1$ holds for all values of $\alpha$. Since $|\IF|>2$, this equality holds for more than two distinct values of $\alpha$, which implies that $\det(T(X))$ must be independent of $\alpha$. Consequently, we conclude that $y_{1,2}=0$.
	
	Now consider the case when $n = 2k \ge 4$ is even. Consider
	$
	X = E_{1,1} + X_1,
	$
	where
	$$
	X_1 =   \alpha E_{2,1}  - \sum_{i> j} y_{i,j} E_{i,j} +  \sum_{j=2}^{n-1} (-1)^j E_{j,j} 
	+ \beta E_{n,n-1} - (\beta - y_{n,n-1})^{-1} E_{n-1,n},
	$$
	with $\alpha, \beta \in \IF$ such that $\beta \neq y_{n,n-1}$.
	Then
	$$
	X =  \begin{psmallmatrix} \tilde X_{11} & \tilde X_{12}\cr
		0 & \tilde X_{22}\cr\end{psmallmatrix} \in \cB_{n},
	$$
	where $\tilde X_{11} \in \cB_{n-2}$ is a lower triangular matrix with diagonal entries $(1,1,-1, \dots, -1,1)$ and
	$$
	\tilde X_{22} =
	\begin{psmallmatrix}
		-1 & -(\beta - y_{n,n-1})^{-1} \cr
		(\beta-y_{n,n-1}) & 0
	\end{psmallmatrix}
	\in \cB_{2}.
	$$
	Since $X \in \cB_{n} \subseteq \cS_n$, we have
	$$
	T(X) = \begin{psmallmatrix} \tilde Z_{11} & \tilde Z_{12} \cr 0 & \tilde Z_{22}\cr\end{psmallmatrix}
	\in \cS_{n}\subseteq \cD_n,
	$$
	where
	$\tilde Z_{11} =
	\begin{psmallmatrix}
		1+ y_{1,1} & y_{1,2} \cr
		\alpha & 1+y_{2,2}
	\end{psmallmatrix}$
	and $\tilde Z_{22}$ is an upper triangular matrix with constant diagonal entries. Thus, we again obtain that $\det(T(X))$ is a linear function of $\alpha$ such that $(\det(T(X)))^2=1$ holds for all $\alpha \in \IF$. Using the same argument as above, we obtain $y_{1,2}=0$.
	
	Therefore $y_{s,t}=0$ for all $1\leq s,t \leq n$ with $s\neq t$.
	
	\textbf{Step 2.}
	Next, we show that $y_{1,1} = \dots = y_{n,n}$. Let $y_{i,i} = y_i$ for all $1\leq i \leq n$.
	
	For $n = 2k+1 \geq 3$, consider $X = X_1  \oplus X_2$, where
	$$
	X_1 =
	\begin{psmallmatrix}
		\alpha & \alpha^2+1\cr
		-1 & -\alpha
	\end{psmallmatrix}
	$$
	and $X_2\in M_{2k-2}(\IF)$ is a diagonal matrix with diagonal entries $(1,-1, 1, \dots, -1,1)$. Since the characteristic polynomial of $X_1$ is $x^2+1$, which is self-reciprocal, we have $X_1 \in \cB_2$. Therefore Proposition~\ref{prop-rev-class} implies that $X \in \cB_n\subseteq \cS_n$, and thus $T(X) \in \cS_n \subseteq \cD_n$. Now write $X = X_0+ E_{n,n}$, where $\tr (X_0) = 0$. Thus
	$$
	T(X) = X_0 + T(E_{n,n}) = X_0 + E_{n,n} +Y=  Z_1 \oplus Z_2,
	$$
	where
	$Z_1 =
	\begin{psmallmatrix}
		\alpha+y_1 & \alpha^2+1 \cr
		-1 & -\alpha+y_2
	\end{psmallmatrix}$
	and
	$Z_2 = \diag(1 + y_3, -1+y_4, \dots, -1+y_{n-1}, 1+y_n)$. Using the same argument as in Step~1, we obtain $y_1 = y_2$.
	
	For $n = 2k \geq 4$, consider $X = X_1\oplus \diag(1,-1, \dots, 1,-1)$, where
	$$
	X_1 =
	\begin{psmallmatrix}
		\alpha+1 & \alpha(\alpha+1) + 1 \cr
		-1 & -\alpha
	\end{psmallmatrix}.
	$$
	Since the characteristic polynomial of $X_1$ is $x^2-x+1$, which is self-reciprocal, we have $X_1 \in \cB_2$. Thus $X \in \cB_n \subseteq \cS_n$ and hence $T(X) \in \cS_n \subseteq \cD_n$. Now write $X = E_{1,1} + X_0$, where $\tr (X_0) =0$. Therefore,
	$$
	T(X) = T(E_{1,1}) + X_0 = Z_1 \oplus Z_2,
	$$
	where
	$Z_1 =
	\begin{psmallmatrix}
		\alpha+1+y_1 & \alpha(\alpha+1) +1 \cr
		-1 & -\alpha+y_2
	\end{psmallmatrix}$
	and $Z_2$ is a diagonal matrix with constant entries. Using the same argument as in Step~1, we again obtain $y_1 = y_2$.
	
	Therefore $y_1 = y_2$. Repeating the argument shows that $y_j = y_1$ for all $1 \leq j \leq n$.
	
	\textbf{Step 3.}
	Finally, we prove that $Y=0$. From the previous steps, we may assume that $Y = y I$ for some $y \in \IF$.
	
	For $n = 2k+1 \geq 3$, consider the trace-zero matrices
	$$
	X_1=\begin{psmallmatrix}0&0&0\\0&0&1\\0&-1&0\end{psmallmatrix}\oplus \sum_{i=4}^n(-1)^i E_{i,i}
	\quad
	\text{and}
	\quad
	X_2=\begin{psmallmatrix}1&1&0\\-1&0&0\\0&0&-1\end{psmallmatrix}\oplus \sum_{i=4}^n(-1)^i E_{i,i}.
	$$
	Then Proposition~\ref{prop-rev-class} implies that $E_{1,1} +X_1$ and $ E_{1,1} +X_2$ belong to $\cB_{n} \subseteq \cS_n$. Therefore
	$$
	T(E_{1,1} +X_1)=\begin{psmallmatrix}1+y&0&0\\0&y&1\\0&-1&y\end{psmallmatrix}\oplus Z
	\quad
	\text{and}
	\quad
	T(E_{1,1} +X_2)=\begin{psmallmatrix}2+y&1&0\\-1&y&0\\0&0&-1+y\end{psmallmatrix}\oplus Z
	$$
	also belong to $\cS_{n} \subseteq \cD_n$, where $Z=\sum_{i=4}^n(-1)^i E_{i,i}$. Thus
	$$
	\det(T(E_{1,1}+X_1))=(y+1)(y^2+1)\beta
	\quad\text{and}\quad
	\det(T(E_{1,1}+X_2))=(y+1)^2(-1+y)\beta
	$$
	both lie in $\{\pm 1\}$,	where $\beta =\det(Z)$. Since $y \neq \pm1$ and $\beta \neq 0$, it follows that
	$$
	(y+1)(y^2+1)=\pm (y+1)^2(-1+y),
	$$
	which implies $y^2+1=\pm(-1+y^2)$. Since $\mathrm{char}(\IF) \neq 2$, we obtain $y=0$.
	
	For $n = 2k \ge 4$, consider the trace-zero matrices
	$
	X_1=
	\begin{psmallmatrix}
		0&1\\
		-1&0
	\end{psmallmatrix}
	\oplus
	\begin{psmallmatrix}
		1&0\\
		0&-1
	\end{psmallmatrix}
	\oplus
	\sum_{i=5}^n (-1)^i E_{i,i}
	$
	and
	$
	X_2=
	\begin{psmallmatrix}
		0&1\\
		-1&0
	\end{psmallmatrix}
	\oplus
	\begin{psmallmatrix}
		0&1\\
		-1&0
	\end{psmallmatrix}
	\oplus
	\sum_{i=5}^n (-1)^i E_{i,i}
	$
	in $M_n(\IF)$. Then Proposition~\ref{prop-rev-class} implies that $E_{1,1}+X_1$ and $E_{1,1}+X_2$ belong to $\cB_n \subseteq \cS_n$. Therefore
	$$
	T(E_{1,1}+X_1)=
	\begin{psmallmatrix}
		1+y&1\\
		-1&y
	\end{psmallmatrix}
	\oplus
	\begin{psmallmatrix}
		1+y&0\\
		0&-1+y
	\end{psmallmatrix}
	\oplus Z
	\quad
	\text{and}
	\quad
	T(E_{1,1}+X_2)=
	\begin{psmallmatrix}
		1+y&1\\
		-1&y
	\end{psmallmatrix}
	\oplus
	\begin{psmallmatrix}
		y&1\\
		-1&y
	\end{psmallmatrix}
	\oplus Z
	$$
	also belong to $\cS_n \subseteq \cD_n$, where
	$
	Z=\sum_{i=5}^n\bigl((-1)^i+y\bigr)E_{i,i}.
	$
	Thus
	$$
	\det(T(E_{1,1}+X_1))=(y^2-1)\beta
	\quad \text{and} \quad
	\det(T(E_{1,1}+X_2))=(y^2+1)\beta,
	$$
	both lie in $\{\pm 1\}$, where $\beta=(y^2+y+1)\det(Z)$. Using the same argument as in the previous case, we conclude that $y=0$.
	
	Hence $Y=0$, and therefore $T(E_{i,i})=E_{i,i}$ for all $1\le i\le n$. This completes the proof.
\end{proof}

Before proving our main theorem, we observe that unlike the case $n \ge 3$, for $n=2$ there exists a unital map $T$ that preserves $\cS_2$ and satisfies the conditions of Lemma~\ref{lem-main-all-unital-scalar} with $r=-1$.
For $A \in M_n(\IF)$, let $\mathrm{tr}(A)$ denote its trace.

\begin{remark}\label{remark-key-3}
	Let $\mathrm{char}(\IF)\neq 2$, and let
	$T: M_2(\IF) \rightarrow M_2(\IF)$ be a bijective linear map such that $T(I)=I$ and the restriction of $T$ to $\mathfrak{sl}_2(\IF)$ has one of the following forms:
	$$T(X)=-X \quad \text{or} \quad T(X)=-X^t.$$
	We claim that $T(\cS_2)\subseteq\cS_2$.
	
	We consider the case where the restriction of $T$ to $\mathfrak{sl}_2(\IF)$ is given by $T(X)=-X$, since a similar argument applies to the other case. Since $\mathrm{char}(\IF)\neq2$, we have
	$$
	T(X)=\frac{\mathrm{tr}(X)}{2}I-\bigg(X-\frac{\mathrm{tr}(X)}{2}I\bigg)
	=-X+\mathrm{tr}(X)I
	\quad \text{for all } X\in M_2(\IF).
	$$
	Observe that $\det(T(X))=\det(X)$ for all $X\in M_2(\IF)$. Therefore, we conclude that $T(\cD_2)\subseteq\cD_2$. Using \cite[Corollary~2]{Ba}, we have $\cC_2=\cD_2$, and thus $T(\cC_2)\subseteq\cC_2$.
	
	Next, we prove that $T(\cB_{2}) \subseteq \cB_{2}$. Let $A\in\cB_2$. Then $A$ is either similar to a diagonal matrix with eigenvalues in $\{-1,1\}$, or similar to the companion matrices $C_f$ or $C_g$, where $f(x)=x^2-1$ and $g(x)=x^2+ax+1$ are self-reciprocal irreducible monic polynomials over $\IF$. In the first and second cases, $A$ is similar to an involution in $\cB_{2}\cap\mathfrak{sl}_2(\IF)$; therefore, $T(A)=-A\in\cB_2$.
	
	In the third case, up to similarity, $A$ is given by
	$
	A=\begin{psmallmatrix}
		0&-1\\
		1&-a
	\end{psmallmatrix}.
	$
	Then we obtain
	$$
	T(A)=-A+\mathrm{tr}(A)I
	=\begin{psmallmatrix}
		-a&1\\
		-1&0
	\end{psmallmatrix}\in\cB_2.
	$$
	Thus $T(\cB_{2}) \subseteq \cB_{2}$.
	
	Hence $T(\cS_2)\subseteq\cS_2$. This shows that $r=-1$ is possible in Lemma~\ref{lem-main-all-unital-scalar} when $n=2$. Furthermore, note that in this case
	$$
	T(X)=QX^{t}Q^{-1} \quad \text{for all } X\in M_2(\IF),
	$$
	where $Q=\begin{psmallmatrix}
		0&1\\
		-1&0
	\end{psmallmatrix}$.
	Thus, in this case as well, the map $T$ is of the same form as in Theorem~\ref{thm-main}. \qed
\end{remark}

We now apply the above results to prove our main theorem.

\textit{Proof of Theorem~\ref{thm-main}}.
The sufficiency part of the result is straightforward; see also Remark~\ref{remark-key-3} for the case $n=2$. Therefore, we focus on proving the necessity.

Let $\cS_n=\cB_n$, $\cC_n$, or $\cD_n$, and let $T: M_n(\IF)\rightarrow M_n(\IF)$ be a bijective linear map satisfying $T(\cS_n)\subseteq\cS_n$, where $n\ge2$. By Lemma~\ref{lem-main-all-form}, there exist a scalar $r\in\IF$ and an invertible matrix $R\in M_n(\IF)$ such that the restriction of $T$ to $\mathfrak{sl}_n(\IF)$ has one of the following forms:
\begin{equation}\label{eq-1-main-thm-proof}
	X\mapsto rT(I)RXR^{-1} \quad \text{or} \quad X\mapsto rT(I)RX^tR^{-1}.
\end{equation}

To prove the result, we consider the following cases depending on $\cS_n$.

\textbf{Case (a).} Let $\cS_n=\cB_n$. In view of Lemma~\ref{lem-key-nonunital-B}, we have $T(I)=\pm I$. Using Lemma~\ref{lem-main-all-unital-scalar} and Remark~\ref{remark-key-3}, we conclude that there exist $\alpha\in\{-1,1\}$ and an invertible matrix $P$ such that $T(I)=\alpha I$, and the restriction of $T$ to $\mathfrak{sl}_n(\IF)$ has one of the following forms:
\begin{equation*}
	X\mapsto \alpha PXP^{-1} \quad \text{or} \quad X\mapsto \alpha PX^tP^{-1}.
\end{equation*}

Suppose that $n$ is not a multiple of $\mathrm{char}(\IF)$. Then every $Z\in M_n(\IF)$ can be written as $Z=aI+X$ for some $a\in\IF$ and $X\in\mathfrak{sl}_n(\IF)$, and the result follows in this case. Now assume that $n$ is a multiple of $\mathrm{char}(\IF)$. Using Lemma~\ref{lem-for-n-divides-char}, we obtain
\begin{equation*}
	E_{i,i}\mapsto \alpha P E_{i,i} P^{-1}
	\quad \text{or} \quad
	E_{i,i}\mapsto \alpha P E_{i,i}^{t} P^{-1},
\end{equation*}
for all $1\le i\le n$. The proof for this case now follows from the fact that the set $\{E_{i,j}\mid 1\le i,j\le n\}$ forms a basis of $M_n(\IF)$.

\textbf{Case (b).} Let $\cS_n=\cC_n$ and $T(I)=I$. The result in this case follows from a similar argument as in Case~(a).

\textbf{Case (c).} Let $\cS_n=\cD_n$. Applying Lemma~\ref{lem-main-all-unital-scalar} and Remark~\ref{remark-key-3}, we obtain that there exists an invertible matrix $Q\in M_n(\IF)$ such that the restriction of $T$ to $\mathfrak{sl}_n(\IF)$ has one of the following forms:
$$
X\mapsto T(I)Q^{-1}XQ
\quad \text{or} \quad
X\mapsto T(I)Q^{-1}X^tQ.
$$
Let $P=T(I)Q^{-1}$. Since $T(I)\in\cD_n$, we have $T(I)=PQ$ with $\det(PQ)=\pm1$, and the restriction of $T$ to $\mathfrak{sl}_n(\IF)$ has one of the following forms:
$$
X\mapsto PXQ
\quad \text{or} \quad
X\mapsto PX^tQ.
$$
We now consider two cases depending on whether $n$ is a multiple of $\char(\IF)$ and use a similar argument as in Case~(a) to complete the proof.
\qed

%Results on (non-unital) linear preservers of $\mathcal C_n$
%Results on (non-unital) linear maps preserving $\mathcal C_n$
%More results related to linear maps preserving $\cC_n$.

\section{Related results}\label{sec-related-results}

This section is divided into two subsections. In the first, we present additional results on the linear preservers of $\cC_n$, and in the second, we provide results for the case when the characteristic of the field is $2$.

\subsection{Non-unital preservers of $\cC_n$}
\label{sec-related-results-part-1}

In this subsection, we discuss the bijective linear maps $T: M_n(\IF) \rightarrow M_n(\IF)$ that map $\cC_n$ into itself. For this purpose, we introduce the following notation.

Define the set
\begin{equation}\label{eq-set-lambda}
	\Lambda_n(\IF) := \{P \in M_n(\IF) \mid P \cC_n \subseteq \cC_n \}.
\end{equation}
Observe that the identity matrix $I$ belongs to $\Lambda_n(\IF)$. This set depends on both $n$ and the underlying field $\IF$. Some of its basic properties are described in the next remark.

\begin{remark}\label{rem-set-Lambda}
	Since $I \in \Lambda_n(\IF)$, we have $\Lambda_n(\IF) \neq \emptyset$ and $\Lambda_n(\IF) \subseteq \cC_n$. In view of \cite{GHR}, $\Lambda_n(\IF)= \cD_n$ when $\cC_n=\cD_n$. Moreover, if $A \in \Lambda_n(\IF)$, then $-A, A^{-1} \in \Lambda_n(\IF)$. This can be shown as follows.
	
	Let $A \in \Lambda_n(\IF)$. Suppose that $X=X_1X_2X_3$ is an arbitrary element of $\cC_n$, where each $X_i^2=I$. Then $-X=(-X_1)X_2X_3$ and $X^{-1}=X_3^{-1}X_2^{-1}X_1^{-1}=X_3X_2X_1$ belong to $\cC_n$. Moreover,
	\[
	A^{-1}X^{-1}A=(A^{-1}X_3A)(A^{-1}X_2A)(A^{-1}X_1A)\in \cC_n.
	\]
	Observe that $-AX=A(-X)$ for all $X\in \cC_n$. Since $A \in \Lambda_n(\IF)$ and $-X\in \cC_n$, we have $-A \in \Lambda_n(\IF)$. Now note that
	\[
	(A^{-1}X)^{-1}=X^{-1}A=A(A^{-1}X^{-1}A).
	\]
	Since $A \in \Lambda_n(\IF)$ and $A^{-1}X^{-1}A \in \cC_n$, it follows that $(A^{-1}X)^{-1} \in \cC_n$, and hence $A^{-1}X \in \cC_n$ for all $X\in \cC_n$. Therefore $A^{-1} \in \Lambda_n(\IF)$. Thus we conclude that $-A, A^{-1} \in \Lambda_n(\IF)$.
	\qed
\end{remark}

The following result characterizes the bijective linear preservers of $\cC_n$ in terms of $\Lambda_n(\IF)$.

\begin{theorem}\label{thm-non-unital-C_n}
	Let $\IF$ be a field with characteristic not equal to $2$. Let $T: M_n(\IF) \rightarrow M_n(\IF)$ be a bijective linear map. Then $T(\cC_n) \subseteq \cC_n$ if and only if there exist invertible matrices $P, Q\in M_n(\IF)$ with $PQ\in \Lambda_n(\IF)$ such that $T$ has the form
	\begin{equation*}
		X \mapsto PXQ \quad \text{or} \quad X \mapsto PX^tQ.
	\end{equation*}
\end{theorem}

\begin{proof}
	Since $PXQ=PQ(Q^{-1}XQ)$, $PX^tQ=PQ(Q^{-1}X^tQ)$, and $Q\cC_nQ^{-1}\subseteq\cC_n$, the sufficiency of the result is straightforward. We now focus on proving the necessity. From \cite[Corollary~2]{Ba}, we have $\cC_2=\cD_2$, so the case $n=2$ follows directly from Remark~\ref{rem-set-Lambda} and Theorem~\ref{thm-main}.
	
	Now assume $n\ge3$. Let $T: M_n(\IF) \rightarrow M_n(\IF)$ be a bijective linear map satisfying $T(\cC_n) \subseteq \cC_n$. Let $A:=T(I) \in \cC_n$. Using Lemma~\ref{lem-main-all-form} and Lemma~\ref{lem-main-all-unital-scalar}, we obtain that there exists an invertible matrix $Q \in M_n(\IF)$ such that the restriction of $T$ to $\mathfrak{sl}_n(\IF)$ has one of the following forms:
	\[
	X\mapsto AQ^{-1}XQ \quad \text{or} \quad X\mapsto AQ^{-1}X^{t}Q.
	\]
	
	We now consider two cases depending on whether $n$ is a multiple of $\char(\IF)$. Using Lemma~\ref{lem-for-n-divides-char} and a similar argument as in Theorem~\ref{thm-main}, we conclude that $T$ has one of the following forms:
	\begin{equation*}
		X \mapsto AQ^{-1}XQ \quad \text{or} \quad X \mapsto AQ^{-1}X^tQ.
	\end{equation*}
	
	Since $T(\cC_n) \subseteq \cC_n$ and $Q^{-1}\cC_nQ\subseteq \cC_n$, it follows that $AX\in \cC_n$ for all $X\in \cC_n$, that is, $A\in\Lambda_n(\IF)$. The proof now follows by setting $P=AQ^{-1}$.
\end{proof}

Note that the above theorem depends on the subset $\Lambda_n(\IF)$ of $\cC_n$. Therefore, a complete characterization of the preservers of $\cC_n$ requires a thorough understanding of its structure. If $\cC_n=\cD_n$, then $\Lambda_n(\IF)=\cD_n$, and hence Theorem~\ref{thm-non-unital-C_n} is equivalent to Theorem~\ref{thm-main}(iii). In this case, the bijective linear preservers of $\cC_n$ are precisely those described in Theorem~\ref{thm-main}(iii). 

Now suppose that $\cC_n \subsetneq \cD_n$. In this case we have $\Lambda_n(\IF) \subseteq \cC_n \subsetneq \cD_n$. In the following remark we examine $\Lambda_3(\IF)$ in certain cases where $\cC_3 \subsetneq \cD_3$ and show that $\Lambda_3(\IF)=\{\pm I\}$.

\begin{remark}\label{rem-struct-lamba-three}
	Let $\mathrm{char}(\IF) \notin \{2,3\}$, and suppose that the polynomial $t^2+t+1$ is reducible in $\IF[t]$. Using Remark~\ref{rem_C_n}, we have $\cC_3 \subsetneq \cD_3$. Moreover, according to \cite[p.~57]{Ba}, there exists a scalar matrix $\alpha I \in \cD_3$ that is not in $\cC_3$, where $\alpha^4+\alpha^2+1=0$. We now show that if $A \in \Lambda_3(\IF)$, then $A=\pm I$.
	
	Suppose that $A \in \Lambda_3(\IF)$ and $A \neq \pm I$. Since $\Lambda_3(\IF) \subseteq \cC_3$, \cite[Corollary~3]{Ba} implies that $A$ is a nonscalar invertible matrix in $\cD_3$. Let $B:=\alpha A^{-1}=A^{-1}(\alpha I)\in M_3(\IF)$. Since $A$ and $\alpha I$ belong to $\cD_3$, it follows that $B$ is a nonscalar matrix in $\cD_3$. Then, by \cite[Corollary~3(b)]{Ba}, we have $B \in \cC_3$. Since $A \in \Lambda_3(\IF)$, we obtain $AB=\alpha I \in \cC_3$, which is a contradiction. Hence $\Lambda_3(\IF)=\{\pm I\}$.
	\qed
\end{remark}

Recall that if $\mathrm{char}(\IF)=3$ or $t^2+t+1$ is irreducible in $\IF[t]$, then $\cC_3=\cD_3$ (see Remark~\ref{rem_C_n}), and thus the bijective linear preservers of $\cC_3$ follow directly from Theorem~\ref{thm-main}(iii). The following result characterizes the bijective linear preservers of $\cC_3$ in the remaining cases, where $\cC_3 \subsetneq \cD_3$ and $\mathrm{char}(\IF) \neq 2$; see also Theorem~\ref{thm-main-order-3-char-2}.

\begin{theorem}\label{thm-prod-3-not2}
	Let $\char(\IF)\neq 2$, and let $T: M_3(\IF) \rightarrow M_3(\IF)$ be a bijective linear map. Suppose that $\cC_3$ is a proper subset of $\cD_3$. Then $T(\cC_3) \subseteq \cC_3$ if and only if there exist a scalar $\alpha \in \{-1,1\}$ and an invertible matrix $P\in M_3(\IF)$ such that $T$ has the form
	$$X \mapsto \alpha PXP^{-1} \quad \text{or} \quad X \mapsto \alpha PX^tP^{-1}.$$
\end{theorem}

\begin{proof}
	Using Remark~\ref{rem_C_n}, we obtain that $\mathrm{char}(\IF) \neq 3$ and that the polynomial $t^2+t+1$ is reducible in $\IF[t]$. Then Remark~\ref{rem-struct-lamba-three} implies that $\Lambda_3(\IF)=\{\pm I\}$. The result now follows from Theorem~\ref{thm-non-unital-C_n}.
\end{proof}

%\subsection{Results for fields with characteristic $2$
	
\subsection{The characteristic $2$ case}\label{sec-related-results-part-2}

In the previous discussion, we assumed that $\mathrm{char}(\IF)\neq 2$. The characterization of linear preservers of products of involutions becomes more challenging when $\mathrm{char}(\IF)=2$. We now discuss the cases $n=2$ and $n=3$ under the assumption that $\mathrm{char}(\IF)=2$. The following remark shows that when $n=2=\mathrm{char}(\IF)$, we have $\cB_2=\cC_2=\cD_2$.

\begin{remark}\label{rem-char-two}
	Let $n=2=\mathrm{char}(\IF)$. Since $\mathrm{char}(\IF)=2$, we have
	$$ \begin{psmallmatrix}
		1 & x\\
		0 & 1
	\end{psmallmatrix}
	\begin{psmallmatrix}
		0 & 1\\
		1 & x
	\end{psmallmatrix}
	=
	\begin{psmallmatrix}
		x & 1+x^2\\
		1 & x
	\end{psmallmatrix}
	=
	\begin{psmallmatrix}
		x & 1\\
		1 & 0
	\end{psmallmatrix}
	\begin{psmallmatrix}
		1 & x\\
		0 & 1
	\end{psmallmatrix}.
	$$
	Thus,
$
	\begin{psmallmatrix}
		0 & 1\\
		1 & x
	\end{psmallmatrix}
	\in \cB_2(\IF)
	\quad \text{for all } x\in\IF.
$
	Consequently, every matrix in $M_2(\IF)$ with determinant $\pm1$ belongs to $\cB_2$, implying that $\cB_2=\cC_2=\cD_2$.
	\qed
\end{remark}

To characterize the bijective linear preservers of $\cS_2$ when $\mathrm{char}(\IF)=2$, we require the following analogue of Lemma~\ref{lem-for-n-divides-char}.

\begin{lemma}\label{lem-char-two}
	Let $\mathrm{char}(\IF)=2$, and let $T: M_2(\IF)\rightarrow M_2(\IF)$ be a bijective linear map such that $T(\cS_2)\subseteq \cS_2$. Suppose that there exists an invertible matrix $P\in M_2(\IF)$ such that the restriction of $T$ to $\mathfrak{sl}_2(\IF)$ has one of the following forms:
	\begin{equation*}
		X \mapsto T(I)PXP^{-1} \quad \text{or} \quad X \mapsto T(I)PX^tP^{-1}.
	\end{equation*}
	Then there exists $Y\in\{O,I\}\subseteq M_2(\IF)$ such that
	$$
	T(E_{i,i})=T(I)P(E_{i,i}+Y)P^{-1} \quad \text{for all } 1\le i\le2.
	$$
\end{lemma}

\begin{proof}
	We may assume that $T(I)=I$ and $T(X)=X$ for all $X\in\mathfrak{sl}_2(\IF)$, since the remaining cases can be handled using a similar argument. It suffices to show that there exists $Y\in\{O,I\}\subseteq M_2(\IF)$ such that
	$$
	T(E_{i,i})=E_{i,i}+Y \quad \text{for all } 1\le i\le2.
	$$
	
	Using Remark~\ref{rem-char-two}, we have $\cS_2=\cD_2$. Hence we assume $T(\cD_2)\subseteq\cD_2$. Let $T(E_{i,i})=E_{i,i}+Y_i$ for all $1\le i\le2$. Since $T(E_{1,1}-E_{2,2})=E_{1,1}-E_{2,2}$, we may assume that
	$$
	T(E_{1,1})=E_{1,1}+Y \quad \text{and} \quad T(E_{2,2})=E_{2,2}+Y,
	$$
	where
	$$
	Y=
	\begin{psmallmatrix}
		y_{1,1} & y_{1,2}\\
		y_{2,1} & y_{2,2}
	\end{psmallmatrix}
	\in M_2(\IF).
	$$
	
	We now show that $Y$ is a scalar matrix. Suppose that $y_{1,2}\neq0$ or $y_{2,1}\neq0$. Without loss of generality, assume $y_{1,2}\neq0$, as a similar argument applies to the case $y_{2,1}\neq0$. Consider the trace-zero matrix
	$$
	X=
	\begin{psmallmatrix}
		0 & y_{1,2}\\
		-1/y_{1,2} & 0
	\end{psmallmatrix}.
	$$
	Then Remark~\ref{rem-char-two} implies $X+E_{1,1},X+E_{2,2}\in\cD_2$. Since $\mathrm{char}(\IF)=2$ (i.e., $-1=1$), we have
	$$
	T(X+E_{1,1})=
	\begin{psmallmatrix}
		1+y_{1,1} & 0\\
		y_{2,1}-1/y_{1,2} & y_{2,2}
	\end{psmallmatrix},
	\quad
	T(X+E_{2,2})=
	\begin{psmallmatrix}
		y_{1,1} & 0\\
		y_{2,1}-1/y_{1,2} & 1+y_{2,2}
	\end{psmallmatrix}.
	$$
	
Using $T(X+E_{i,i})\in\cD_2$ for $1\le i\le2$, we obtain
	$$
	\det(T(X+E_{1,1}))=(1+y_{1,1})y_{2,2}=1
	\quad \text{and} \quad
	\det(T(X+E_{2,2}))=y_{1,1}(1+y_{2,2})=1.
	$$
	It follows that $y_{1,1}=y_{2,2}$ and $(1+y_{1,1})y_{1,1}=1$, i.e., $y_{1,1}^2+y_{1,1}+1=0$. Since $\mathrm{char}(\IF)=2$, the polynomial $t^2+t+1$ is irreducible over $\IF$, which yields a contradiction. Therefore $y_{1,2}=y_{2,1}=0$, and hence
	$$
	Y=
	\begin{psmallmatrix}
		y_{1,1} & 0\\
		0 & y_{2,2}
	\end{psmallmatrix}.
	$$
	
	Next consider the trace-zero matrix
	$$
	X=
	\begin{psmallmatrix}
		0 & 1\\
		1 & 0
	\end{psmallmatrix}.
	$$
	Using Remark~\ref{rem-char-two}, we have $X+E_{1,1},X+E_{2,2}\in\cD_2$. Thus
	$$
	T(X+E_{1,1})=
	\begin{psmallmatrix}
		1+y_{1,1} & 1\\
		1 & y_{2,2}
	\end{psmallmatrix}\in\mathcal D_2,
	\qquad
	T(X+E_{2,2})=
	\begin{psmallmatrix}
		y_{1,1} & 1\\
		1 & 1+y_{2,2}
	\end{psmallmatrix}
	\in \cD_2.
	$$
Since $-1=1$, we obtain
	$$
	\det(T(X+E_{1,1}))=(1+y_{1,1})y_{2,2}+1=1
	\quad \text{and} \quad
	\det(T(X+E_{2,2}))=y_{1,1}(1+y_{2,2})+1=1.
	$$
	This implies $y_{1,1}=y_{2,2}$ and $(1+y_{1,1})y_{1,1}=0$. Hence $y_{1,1}=0$ or $1$, and therefore $Y\in\{O,I\}$. This completes the proof.
\end{proof}

The following result classifies the bijective linear preservers of $\cS_{2}$ when $\mathrm{char}(\IF)=2$.

\begin{theorem}\label{thm-main-2}
	Let $\mathrm{char}(\IF)=2$, and let $T: M_2(\IF) \rightarrow M_2(\IF)$ be a bijective linear map. Then $T(\cS_{2}) \subseteq \cS_{2}$ if and only if there exist invertible matrices $P, Q\in M_2(\IF)$ with $\det(PQ)=1$ such that $T$ has the form
	$$
	X \mapsto PXQ \quad \text{or} \quad X \mapsto PX^tQ.
	$$
\end{theorem}

\begin{proof}
	The converse is straightforward and follows from Remark~\ref{rem-char-two}. Therefore, we focus on proving the necessity. By Remark~\ref{rem-char-two}, we have $\cB_2=\cC_2=\cD_2$. Hence $T(\cD_{2}) \subseteq \cD_{2}$. Let $A := T(I)$.
	
	Since $\mathrm{char}(\IF)=2$ is a multiple of $n=2$, Lemma~\ref{lem-main-all-form} implies that there exists an invertible matrix $S \in M_2(\IF)$ such that the restriction of $T$ to $\mathfrak{sl}_{2}(\IF)$ has one of the following forms:
	\begin{equation*}
		X \mapsto ASXS^{-1} \quad \text{or} \quad X \mapsto ASX^tS^{-1}.
	\end{equation*}
	
	Without loss of generality, we consider the first case, as the second case can be handled by a similar argument. In view of Lemma~\ref{lem-char-two}, there exists $Y\in\{O,I\}$ such that
	$$
	T(E_{i,i})=AS(E_{i,i}+Y)S^{-1} \quad \text{for all } 1\le i\le2.
	$$
	
	Observe that every matrix 
	$
	Z=\begin{psmallmatrix}
		z_{1,1} & z_{1,2}\\
		z_{2,1} & z_{2,2}
	\end{psmallmatrix}
	\in M_2(\IF)
	$
	can be written as
	\begin{equation}\label{eq-char-2}
		Z = z_{1,1}E_{1,1} + z_{2,2}E_{2,2} + Z_1,
	\end{equation}
	where 
	$
	Z_1=
	\begin{psmallmatrix}
		0 & z_{1,2}\\
		z_{2,1} & 0
	\end{psmallmatrix}
	$
	is a trace-zero matrix in $M_2(\IF)$. We now consider the following two cases:
	
	\begin{enumerate}[label={\normalfont(\roman*)}]
		\item If $Y=O$, then $T(E_{i,i})=ASE_{i,i}S^{-1}$ for all $1 \le i \le 2$. Hence
		$$
		T(Z)=ASZS^{-1} \quad \text{for all } Z\in M_2(\IF).
		$$
		
		\item If $Y=I$, then
		$
		T(E_{1,1}) = AS(E_{1,1}+I)S^{-1}=ASE_{2,2}S^{-1}
		$
		and
		$
		T(E_{2,2}) = AS(E_{2,2}+I)S^{-1}=ASE_{1,1}S^{-1}.
		$
		Using \eqref{eq-char-2}, we obtain that $T$ has the form
		\begin{equation*}
			\begin{psmallmatrix}
				z_{1,1} & z_{1,2}\\
				z_{2,1} & z_{2,2}
			\end{psmallmatrix}
			\mapsto
			AS
			\begin{psmallmatrix}
				z_{2,2} & z_{1,2}\\
				z_{2,1} & z_{1,1}
			\end{psmallmatrix}
			S^{-1}.
		\end{equation*}
	\end{enumerate}
	
	Consider the invertible matrix $R = E_{1,2} + E_{2,1} \in M_2(\IF)$. Then
	$$
	T(Z)=ASRZ^{t}R^{-1}S^{-1}=A(SR)Z^{t}(SR)^{-1} \quad \text{for all } Z\in M_2(\IF).
	$$
	
	Now choose an invertible matrix $Q \in M_2(\IF)$ depending on $Y$ as follows: $Q=S$ if $Y=O$, and $Q=SR$ if $Y=I$. Then $T$ has the form
	$$
	X\mapsto AQ^{-1}XQ \quad \text{or} \quad X\mapsto AQ^{-1}X^tQ.
	$$
	The proof now follows by taking $P=AQ^{-1}$.
\end{proof}

Next, let $n=3$ and $\char(\IF)=2$. Unlike the case $n=2=\mathrm{char}(\IF)$, here we may have $\cB_3 \subsetneq \cC_3 \subsetneq \cD_3$, as shown in the following remark. We will also classify elements of $\cB_3$ in this remark.

\begin{remark}\label{rem-order-3-char-2}
	Let $\char(\IF)=2$. In view of Proposition~\ref{prop-rev-class} (see also \cite[Theorem 3.1]{AOPV}), an element $A \in M_3(\IF)$ belongs to $\cB_3$ if and only if $A$ is similar to one of the following matrices:
	\begin{equation*}
		I, \quad
		\begin{psmallmatrix}
			1 & 1 & 0\\
			0 & 1 & 0\\
			0 & 0 & 1
		\end{psmallmatrix},
		\quad \text{or} \quad
		\begin{psmallmatrix}
			0 & 0 & 1\\
			1 & 0 & a\\
			0 & 1 & a
		\end{psmallmatrix},
		\quad \text{where } a \in \IF .
	\end{equation*}
	
	Note that the matrix
	\[
	\begin{psmallmatrix}
		0 & 0 & 1\\
		1 & 0 & 0\\
		0 & 1 & 1
	\end{psmallmatrix}
	\notin \cB_3,
	\]
	but it belongs to $\cC_3$ (see \cite[Corollary 3(b)]{Ba}). Suppose that $|\IF|>2$. Then $t^2+t+1$ is reducible in $\IF[t]$, i.e., there exists $\alpha \in \IF$ such that $\alpha^4+\alpha^2+1=0$. From \cite[p.~57]{Ba}, we have $\alpha I\notin\cC_3$ but $\alpha I\in\cD_3$. However, if $|\IF|=2$, then $t^2+t+1$ is irreducible in $\IF[t]$, and by Remark~\ref{rem_C_n} we have $\cC_3=\cD_3$. We summarize these observations as follows:
	\begin{enumerate}[label={\normalfont(\roman*)}]
		\item If $|\IF|>2$ and $\char(\IF)=2$, then $\cB_3 \subsetneq \cC_3 \subsetneq \cD_3$.
		\item If $|\IF|=2$ and $\char(\IF)=2$, then $\cB_3 \subsetneq \cC_3=\cD_3$.\qed
	\end{enumerate}
\end{remark}

To characterize bijective linear preservers of $\cB_3$ and $\cC_3$ over a field $\IF$ with $\mathrm{char}(\IF)=2$, we require the following analogue of Lemma~\ref{lem-key-nonunital-B}.

\begin{lemma}\label{lem-nonunital-C-3}
	Let $\mathrm{char}(\IF)=2$, and let $T: M_3(\IF) \rightarrow M_3(\IF)$ be a bijective linear map such that either $T(\cB_{3}) \subseteq \cB_{3}$, or $T(\cC_{3}) \subseteq \cC_{3}$ and $|\IF|>2$. Suppose that there exist a nonzero scalar $r \in \IF$ and an invertible matrix $P \in M_3(\IF)$ such that the restriction of $T$ to $\mathfrak{sl}_3(\IF)$ has one of the following forms:
	\begin{equation*}
		X \mapsto r T(I)PXP^{-1} \quad \text{or} \quad X \mapsto r T(I)PX^tP^{-1}.
	\end{equation*}
	Then $T(I)=I$.
\end{lemma}

\begin{proof}
	We consider the case $T(X)=rT(I)PXP^{-1}$ for all $X\in\mathfrak{sl}_3(\IF)$, as the second case can be proved by a similar argument. Let $A:=T(I)$. Without loss of generality, assume that $T(X)=rAX$ for all $X\in\mathfrak{sl}_3(\IF)$. To prove the result, we consider the following two cases.
	
	\textbf{Case (a).} Let $T(\cB_{3}) \subseteq \cB_{3}$. Then $A=T(I)\in\cB_{3}$. Suppose that $A\neq I$. Then, by Remark~\ref{rem-order-3-char-2}, $A$ must be one of the matrices
	\[
	\begin{psmallmatrix}
		1 & 1 & 0\\
		0 & 1 & 0\\
		0 & 0 & 1
	\end{psmallmatrix},
	\quad
	\text{or}
	\quad
	\begin{psmallmatrix}
		0 & 0 & 1\\
		1 & 0 & a\\
		0 & 1 & a
	\end{psmallmatrix},
	\]
	where $a\in\IF$. In each of the following subcases, we construct a matrix $X\in\cB_{3}$ such that $T(X)\notin\cB_{3}$, which yields a contradiction.
	
	\textbf{Subcase (i).} Let
$
	A=
	\begin{psmallmatrix}
		1 & 1 & 0\\
		0 & 1 & 0\\
		0 & 0 & 1
	\end{psmallmatrix}.
$
	Consider
$
	X=
	\begin{psmallmatrix}
		0 & 1 & 1\\
		0 & 0 & 1\\
		1 & 1 & 0
	\end{psmallmatrix},
$
	which lies in $\mathfrak{sl}_3(\IF)\cap\cB_{3}$. To verify that $X\in\cB_{3}$, note that $gX=X^{-1}g$, where
$
	g=
	\begin{psmallmatrix}
		1 & 0 & 0\\
		1 & 1 & 0\\
		0 & 0 & 1
	\end{psmallmatrix}.
$
	Since $X\in\cB_{3}$, we obtain
	\[
	T(X)=rAX=
	\begin{psmallmatrix}
		0 & r & 0\\
		0 & 0 & r\\
		r & r & 0
	\end{psmallmatrix}.
	\]
	The characteristic polynomial of $T(X)$ is
	\[
	\det(\lambda I-T(X))=\lambda^3-r^2\lambda-r^3,
	\]
	which is not self-reciprocal. Hence, by Proposition~\ref{prop-rev-class}, we conclude that $T(X)\notin\cB_3$.
	
	\textbf{Subcase (ii).} Let
$
	A=
	\begin{psmallmatrix}
		0 & 0 & 1\\
		1 & 0 & a\\
		0 & 1 & a
	\end{psmallmatrix}
$
	for some $a\in\IF$. Consider $X_b=I+r^{-1}bE_{1,3}\in\cB_3$, where $b\in\IF$. Then
	\[
	T(X)=T(I)+T(r^{-1}bE_{1,3})=A(I+bE_{1,3})=
	\begin{psmallmatrix}
		0 & 0 & 1\\
		1 & 0 & b+a\\
		0 & 1 & a
	\end{psmallmatrix}.
	\]
	The characteristic polynomial of $T(X)$ is
	\[
	\det(\lambda I-T(X))=\lambda^3-a\lambda^2-(a+b)\lambda-1.
	\]
	Choose $b=1$ if $a=0$, and $b=a$ if $a\neq0$. For this choice of $b$, let $X:=X_b$. Then the characteristic polynomial of $T(X)$ is not self-reciprocal, and hence Proposition~\ref{prop-rev-class} implies that $T(X)\notin\cB_3$.
	
	From the above discussion, we conclude that if $A\neq I$, then there exists $X\in\cB_{3}$ such that $T(X)\notin\cB_{3}$, contradicting the assumption that $T(\cB_{3})\subseteq\cB_{3}$. Hence $T(I)=A=I$.
	
	\textbf{Case (b).} Let $T(\cC_{3}) \subseteq \cC_{3}$ and $|\IF|>2$. Then $A=T(I)\in\cC_{3}$ and $t^2+t+1$ is reducible in $\IF[t]$, i.e., there exists $b\in\IF$ such that $b^4+b^2+1=0$. Since $\mathrm{char}(\IF)=2$, it follows from \cite[p.~57]{Ba} that there are exactly two matrices in $\cD_3$ that are not in $\cC_3$, namely matrices of the form $\alpha I$, where $\alpha^4+\alpha^2+1=0$. Furthermore, since $T(\cC_3)\subseteq\cC_3$ and $T$ is bijective, there exists $a\in\IF$ with $a^4+a^2+1=0$ such that $T(aI)=bI$. By the linearity of $T$, we obtain $T(I)=(b/a)I$. Using \cite[Corollary 3(a)]{Ba}, we have $(b/a)^2=1$. Since $\mathrm{char}(\IF)=2$, it follows that $b/a=1$, and therefore $T(I)=A=I$. This completes the proof.
\end{proof}

The next result is an analogue of Lemma~\ref{lem-main-all-unital-scalar} for the case $n=3$ and $\mathrm{char}(\IF)=2$.

\begin{lemma}\label{lem-unital-scalar-order-3}
	Let $\mathrm{char}(\IF)=2$, and let $T: M_3(\IF) \rightarrow M_3(\IF)$ be a bijective linear map such that $T(\cS_{3}) \subseteq \cS_{3}$. Suppose that there exist a nonzero scalar $r \in \IF$ and an invertible matrix $P \in M_3(\IF)$ such that the restriction of $T$ to $\mathfrak{sl}_3(\IF)$ has one of the following forms:
	\begin{equation*}
		X \mapsto r T(I)PXP^{-1} \quad \text{or} \quad X \mapsto rT(I)PX^tP^{-1}.
	\end{equation*}
	Then $r=1$.
\end{lemma}

\begin{proof}
	We consider only the first case, as the second follows by a similar argument. Without loss of generality, assume that $T(X)=rT(I)X$ for all $X\in\mathfrak{sl}_3(\IF)$.
	
	Let $A:=T(I)$. Then $A \in \cS_3 \subseteq \cD_3$, and hence $\det(A)=1$. Consider
	$
	X=\begin{psmallmatrix}
		0 & 0 & 1\\
		1 & 0 & r\\
		0 & 1 & r
	\end{psmallmatrix}\in M_3(\IF).
	$
	Using Remark~\ref{rem-order-3-char-2}, we have $X\in\cB_3 \subseteq \cS_3$. Since $\mathrm{char}(\IF)=2$ and $r\neq0$, we can write
	$$
	X=rI+\begin{psmallmatrix}
		r & 0 & 1\\
		1 & r & r\\
		0 & 1 & 0
	\end{psmallmatrix}.
	$$
	This implies that
	$$
	T(X)=A\begin{psmallmatrix}
		r+r^2 & 0 & r\\
		r & r+r^2 & r^2\\
		0 & r & r
	\end{psmallmatrix}.
	$$
	Using $\det(A)=1$, we obtain
	$$
	\det(T(X))=(r^4+2r^2)\det(A)=r^4.
	$$
	Since $X\in \cS_3$, it follows that $T(X)\in \cS_3 \subseteq \cD_3$, and hence $\det(T(X))=r^4=1$. Because $\char(\IF)=2$, the equation $r^4=1$ implies $r=1$.
\end{proof}

The following result generalizes Theorem~\ref{thm-main} to the case $n=3$ and $\char(\IF)=2$, and characterizes the bijective linear preservers of $\cS_3$; cf.~Theorem~\ref{thm-prod-3-not2}.

\begin{theorem}\label{thm-main-order-3-char-2}
	Let $\char(\IF)=2$, and let $T: M_3(\IF) \rightarrow M_3(\IF)$ be a bijective linear map. Then the following statements hold.
	\begin{enumerate}[label={\normalfont(\roman*)}]
		\item $T(\cB_{3}) \subseteq \cB_{3}$, or $T(\mathcal C_{3}) \subseteq \mathcal C_{3}$ and $|\IF|>2$, if and only if there exists an invertible matrix $P\in M_3(\IF)$ such that $T$ has the form
		$$
		X \mapsto PXP^{-1} \quad \text{or} \quad X \mapsto PX^tP^{-1}.
		$$
		
		\item $T(\mathcal C_{3}) \subseteq \mathcal C_{3}$ and $|\IF|=2$, or $T(\mathcal D_{3}) \subseteq \mathcal D_{3}$, if and only if there exist invertible matrices $P,Q\in M_3(\IF)$ with $\det(PQ)=1$ such that $T$ has the form
		$$
		X \mapsto PXQ \quad \text{or} \quad X \mapsto PX^tQ.
		$$
	\end{enumerate}
\end{theorem}

\begin{proof}
	Since $\char(\IF)=2$, Remark~\ref{rem-order-3-char-2} implies that $\cB_3 \subsetneq \cC_3 \subsetneq \cD_3$ for $|\IF|>2$, and $\cB_3 \subsetneq \cC_3=\cD_3$ for $|\IF|=2$. Moreover, every matrix in $M_3(\IF)$ can be written as the sum of a scalar matrix and a trace-zero matrix. The proof now follows from Lemmas~\ref{lem-main-all-form}, \ref{lem-nonunital-C-3}, and \ref{lem-unital-scalar-order-3}, together with a similar argument as in Theorem~\ref{thm-main}.
\end{proof}

\section{Future directions}\label{sec-remarks}

Below we outline several potential directions for future research.

\begin{enumerate}[label={\normalfont(\Alph*)}]
	
	\item 
	It would be interesting to determine the structure of bijective linear maps $T: M_n(\IF)\rightarrow M_n(\IF)$ satisfying $T(\cC_n) \subseteq \cC_n$ in the case where $\cC_n \neq \cD_n$ and $n>3$. In Theorem~\ref{thm-main}, we characterized such maps under the additional assumption that $T$ is unital (i.e., $T(I)=I$). The characterization of the non-unital case is given in Theorem~\ref{thm-non-unital-C_n} in terms of the set $\Lambda_n(\IF)$ defined in \eqref{eq-set-lambda}. 
	
	The structure of this set is examined in Remark~\ref{rem-struct-lamba-three} for the case $\cC_3 \neq \cD_3$, where it is used to derive Theorem~\ref{thm-prod-3-not2}. Based on our investigations, we conjecture that $\Lambda_n(\IF)=\pm I$ whenever $\cC_n \neq \cD_n$. In other words, if $\cC_n \neq \cD_n$, then any bijective linear map $T$ satisfying $T(\cC_n) \subseteq \cC_n$ must satisfy $T(I)=\pm I$.
	
	\item 
	The bijective linear preservers of products of involutions in $M_n(\IF)$, where $\IF$ is a field with $\mathrm{char}(\IF)=2$, were characterized in Theorems~\ref{thm-main-2} and~\ref{thm-main-order-3-char-2} for $n=2$ and $n=3$, respectively. Investigating this problem further for $n>3$ and $\mathrm{char}(\IF)=2$ would be an interesting direction for future research.
	
	\item 
	Another possible direction for further study is to relax the bijectivity assumption in Theorem~\ref{thm-main}.
	
	\item 
	A related research direction arises in the context of $\mathcal{B}(\mathcal{H})$, the algebra of bounded linear operators on an infinite-dimensional complex Hilbert space $\mathcal{H}$. Halmos and Kakutani (see \cite{HK}) showed that every unitary in $\mathcal{B}(\mathcal{H})$ can be expressed as a product of four symmetries (i.e., unitary involutions), and bijective continuous linear maps on $\mathcal{B}(\mathcal{H})$ preserving such products have been classified in \cite[Corollary~1]{Ra}. This naturally leads to the problem of classifying bijective continuous linear maps on $\mathcal{B}(\mathcal{H})$ that preserve products of two or three symmetries. More broadly, analogous preserver problems arise in the setting of factors (von Neumann algebras with trivial center), motivated by Fillmore’s results for type ${\rm II}_\infty$ and ${\rm III}$ factors (see \cite{Fi}), and recent work by Bhat et al.\ on products of symmetries in type ${\rm II}_1$ factors (see \cite{BNS}).
	
	\item 
	In this paper, we investigated the problem of characterizing linear preservers of the set of products of two involutions (i.e., strongly reversible or bireflectional elements) in $M_n(\IF)$. The classification of such elements has been studied extensively for many classical matrix groups; see \cite{OS} for a survey. This naturally leads to the question of characterizing bijective linear maps on $M_n(\IF)$ that preserve these sets.
	
\end{enumerate}

\bigskip
\noindent
{\bf \Large Acknowledgements}\nopagebreak

The authors thank the referee for a careful reading of the manuscript and for insightful comments and suggestions that significantly improved the quality and presentation of the paper. The authors are grateful to Rohit Dilip Holkar (IISER Bhopal) for organizing the ``2024 AESIM School on Linear Preserver Problems,'' where this project began. They also thank Allen Herman (University of Regina) for valuable insights and discussions. 

C.-K. Li is an affiliate member of the Institute for Quantum Computing, University of Waterloo; his research is supported by the Simons Foundation (Grant No.\ 851334). T. Lohan acknowledges financial support from the IIT Kanpur Postdoctoral Fellowship. S. Singla’s research is supported by the University of Regina’s Faculty of Science and by the Pacific Institute for the Mathematical Sciences (PIMS) Postdoctoral Fellows Program.

\medskip\noindent
(Li) Department of Mathematics, College of William \& Mary, Williamsburg, VA 23187, USA.
E-mail: ckli@math.wm.edu

\medskip\noindent
(Lohan)
Department of Mathematics and Statistics, Indian Institute of Technology Kanpur, Kanpur 208016, India.
E-mail: tejbirlohan70@gmail.com

\medskip\noindent
(Singla) 
The Department of Mathematics and Statistics, University of Regina, Canada. 
E-mail: ss774@snu.edu.in


\begin{thebibliography}{WW}
	
	\bibitem{AOPV} Avni	N.;	Onn,   U.; Prasad, A.; Vaserstein, L. : Similarity classes of $3 \times 3$ matrices over a local principal ideal ring. \textit{Comm. Algebra} 37 (2009), no. 8, 2601--2615.
	
	\bibitem{Ba} Ballantine, C. S. : Products of lnvolutory Matrices. I. \textit{Linear Multilinear Algebra} 5 (1977/78), no. 1, 53--62.
	
	
		\bibitem{BNS} Bhat, B. V. R.; Nayak, S.; Shankar, P. : On products of symmetries in von Neumann algebras.
	J\textit{. Operator Theory} 92 (2024), no. 2, 579--596.
	
	

	
	\bibitem{BPW} 
	Botta, P., Pierce, S.; Watkins, W. : Linear transformations that preserve the nilpotent matrices. 
	\textit{Pacific J. Math.} 104 (1983), 39--46.
	
	
	\bibitem{Dj}  Djokovi\'{c}, D. \v{Z}. : Product of two involutions. \textit{Arch. Math. (Basel)} 18 (1967), 582--584.
	
		\bibitem{Fi}	Fillmore, P. A. : On products of symmetries.
	\textit{Canadian J. Math.} 18 (1966), 897--900.
	

	
	\bibitem{Ge} Gerstenhaber	M. : On nilalgebras and linear varieties of nilpotent matrices I, \textit{Amer. J. Math.} 80 (1958) 614--622. 
	
	\bibitem{Gu} Gustafson, W. H. : On products of involutions. \textit{Paul Halmos}. Springer-Verlag, New York, 1991, 237--255. 
	
	\bibitem{GHR} Gustafson, W. H.; Halmos, P. R.; Radjavi, H. :
	Products of involutions.
	\textit{Linear Algebra Appl.} 13 (1976), 157--162.
	
	\bibitem{GLS}  Guterman, A.; Li, C.-K.; \v{S}emrl, P.
	: Some general techniques on linear preserver problems.
	\textit{Linear Algebra Appl.} 315 (2000), no. 1--3, 61--81.
	
	\bibitem{HK} Halmos, P. R.; Kakutani, S. : 	Products of symmetries.
	\textit{Bull. Amer. Math. Soc.} 64 (1958), 77--78.
	
	\bibitem{HP}
	Hoffman, F.; Paige, E. C. :
	Products of Two Involutions in the General Linear Group.
	\textit{Indiana Univ. Math. J.}   20 (1970/71), 1017--1020.
	
	
	
	\bibitem{LP} Li, C.-K.; Pierce, S. : Linear preserver problems. \textit{Amer. Math. Monthly} 108 (2001), no. 7, 591--605.
	
	
	
	\bibitem{LT} Li, C.-K.; Tsing, N.-K.  : Linear preserver problems: a brief introduction and some special techniques.
	Directions in matrix theory (Auburn, AL, 1990),
	\textit{Linear Algebra Appl.}  162/164 (1992), 217–235.
	
	
	
	\bibitem{LTWW}
	Li, C.-K.;  Tsai, M.-C.; Wang, Y.-S.;   Wong, N.-C. : Linear maps preserving matrices annihilated by a fixed polynomial. \textit{Linear Algebra Appl.} 674 (2023), 46--67.
	
	\bibitem{Liu}  Liu, K. M. : Decomposition of matrices into three involutions.	\textit{Linear Algebra Appl.} 111 (1988), 1--24.
	
	\bibitem{MOR} 	Mathes, B.; Omladič, M.; Radjavi, H. :
	Linear spaces of nilpotent matrices.
	\textit{Linear Algebra Appl.} 149 (1991), 215--225.
	

	
	\bibitem{OS} 
	O'Farrell, A.  G.;  Short, I. : {\it Reversibility in Dynamics and Group Theory}, London Mathematical Society Lecture Note Series, vol.416, Cambridge University Press, Cambridge, 2015.
	
		\bibitem{Ra}	Raïs, M. : The unitary group preserving maps (the infinite-dimensional case).
	\textit{Linear and Multilinear Algebra} 20 (1987), no. 4, 337--345.
	
		\bibitem{Se} Serežkin, V. N. :
	Linear transformations preserving nilpotency.
	\textit{Vestsī Akad. Navuk BSSR Ser. Fīz.-Mat. Navuk} 1985, no. 6, 46--50, 125.
	
	\bibitem{Wo}  Wonenburger, M.  J. : Transformations which are products of two involutions. {\it J. Math. Mech.}  16 (1966), 327--338.
	
	\bibitem{ZTC} 
Zhang,	X.;  Tang, X.; Cao,  C. :
	\textit{Preserver Problems on Spaces of Matrices},
	Science Press, Beijing, 2007, 409 pp., ISBN 9787030189318.
	
%	\bibitem{Zhan} Zhan, X. : \textit{Matrix theory}, Grad. Stud. Math., 147 American Mathematical Society, Providence, RI, 2013. x+253 pp.
	
\end{thebibliography}
\end{document}